\documentclass[reqno, 11pt]{amsart}%
\usepackage{amssymb}
\usepackage[nesting]{hyperref}
\usepackage[pdftex]{graphicx}
\usepackage{listings}
\usepackage{multirow}
\usepackage{placeins}
\usepackage{color}
\usepackage{subfigure}
\usepackage{lscape}
\usepackage{dsfont}
\usepackage{amsmath}
\usepackage{amsfonts}
\usepackage{tikz}
\usepackage{verbatim}
\usepackage{accents}
\usepackage{bbm} 

\newcommand{\BlackBoxes}{\global\overfullrule5pt}

\BlackBoxes

\newcommand{\R}{\mathbb{R}}
\newcommand{\N}{\mathbb{N}}

\newcommand{\Pop}{\mathbb{P}}
\newcommand{\Q}{\mathbb{Q}}

\newcommand{\A}{\mathcal{A}}
\newcommand{\B}{\mathcal{B}}
\newcommand{\BB}{\mathbb{B}}
\renewcommand{\b}{\mathbf{b}}

\newcommand{\E}{\mathbb{E}}

\renewcommand{\H}{\mathcal{H}}

\newcommand{\M}{\mathcal{M}}

\renewcommand{\P}{\mathbb{P}}

\newcommand{\Qc}{\mathcal{Q}}
\newcommand{\Qf}{\mathfrak{Q}}

\newcommand{\T}{\mathcal{T}}

\newcommand{\Z}{\mathcal{Z}}
\newcommand{\Zf}{\mathfrak{Z}}
\newcommand{\1}{\mathbbm{1}}
\newcommand{\q}{F^{-1}}

\DeclareMathOperator{\dif}{d}
\DeclareMathOperator{\ES}{ES}

\DeclareMathOperator{\VaR}{VaR}
\newcommand{\ubar}[1]{\underaccent{\bar}{#1}}

\newtheorem{theorem}{Theorem}
\newtheorem{corollary}[theorem]{Corollary}
\newtheorem{lemma}[theorem]{Lemma}
\newtheorem{proposition}[theorem]{Proposition}
\theoremstyle{definition}
\newtheorem{example}[theorem]{Example}
\newtheorem{remark}[theorem]{Remark}
\newtheorem{definition}[theorem]{Definition}
\newtheorem{assumption}[theorem]{Assumption}

\numberwithin{equation}{section} \numberwithin{theorem}{section}
\def\0{\kern0pt\-\nobreak\hskip0pt\relax}

\makeatletter
\AtBeginDocument{ \def\@serieslogo{ \vbox to\headheight{ \parindent\z@ \fontsize{6}{7\p@}\selectfont
\vss}}}

\def\makeoverbar#1#2#3#4#5#6#7{ \setbox0=\hbox{$\m@th#2\mkern#5mu{{}#3{}}\mkern#6mu$} \setbox1=\null \dimen@=#4\fontdimen8#13 \dimen@=3.5\dimen@
\advance\dimen@ by \ht0 \dimen@=-#7\dimen@ \advance\dimen@ by \wd0
\ht1=\ht0 \dp1=\dp0 \wd1=\dimen@
\dimen@=\fontdimen8#13 \fontdimen8#13=#4\fontdimen8#13
\rlap{\hbox to \wd0{$\m@th\hss#2{\overline{\box1}}\mkern#5mu$}}
\fontdimen8#13=\dimen@}
\def\mylabel#1#2{{\def\@currentlabel{#2}\label{#1}}}
\makeatother
\begin{document}
\title[Markov Decision Processes with Recursive Risk Measures]{Markov Decision Processes with Recursive Risk Measures}
\author[N. \smash{B\"auerle}]{Nicole B\"auerle}
\address[N. B\"auerle]{Department of Mathematics,
Karlsruhe Institute of Technology (KIT), D-76128 Karlsruhe, Germany}

\email{\href{mailto:nicole.baeuerle@kit.edu}{nicole.baeuerle@kit.edu}}

\author[A. \smash{Glauner}]{Alexander Glauner}
\address[A. Glauner]{Department of Mathematics,
Karlsruhe Institute of Technology (KIT), D-76128 Karlsruhe, Germany}

\email{\href{mailto:alexander.glauner@kit.edu} {alexander.glauner@kit.edu}}


\begin{abstract}
In this paper, we consider risk-sensitive Markov Decision Processes (MDPs) with Borel state and action spaces and unbounded cost under both finite and infinite planning horizons. Our optimality criterion is based on the recursive application of static risk measures. This is motivated by recursive utilities in the economic literature, has been studied before for the entropic risk measure and is extended here to an axiomatic characterization of suitable risk measures. We derive a Bellman equation and prove the existence of Markovian optimal policies. For an infinite planning horizon, the model is shown to be contractive and the optimal policy to be stationary. Moreover, we establish a connection to distributionally robust MDPs, which provides a global interpretation of the recursively defined objective function. Monotone models are studied in particular. 
\end{abstract}
\maketitle


\makeatletter \providecommand\@dotsep{5} \makeatother

\vspace{0.5cm}
\begin{minipage}{14cm}
{\small
\begin{description}
\item[\rm \textsc{ Key words}]
{\small Risk-Sensitive Markov Decision Process; Risk Measure; Robustness}
\item[\rm \textsc{AMS subject classifications}] 
{\small 90C40, 91G70 }

\end{description}
}
\end{minipage}

\section{Introduction}

In this paper, we extend Markov Decision Processes (MDPs) to a recursive application of static risk measures. Our framework is such that it is applicable for a wide range of practical models. In particular we consider Borel state and action spaces, unbounded cost functions and rather general risk measures. 

In standard MDP theory we are concerned with minimizing the expected discounted cost of a controlled dynamic system over a finite or infinite time horizon.  The expectation has the nice property that it can be iterated which yields a recursive solution theory for these kind of problems, see e.g. the textbooks by \cite{puterman2014markov,hernandez2012discrete,BaeuerleRieder2011} for a mathematical treatment. However, there are applications where the simple expectation, which does not reflect the true risk of a decision, might not be the best choice to evaluate decisions. In particular when the management of cash flows is concerned, economists prefer to use dynamic utilities to compare their performance. An early axiomatic treatment of a dynamic utility which takes into account the revealed information is \cite{KrepsPorteus1978}. 
Later, the focus was more on an extension of static risk measures to dynamic risk measures. We mention here the following axiomatic approaches \cite{EpsteinSchneider2003,riedel2004dynamic,detlefsen2005conditional,weber2006distribution} just to name some of them. These approaches do not consider a control. For an overview up to 2011 see \cite{acciaio2011dynamic}. Later, besides the axiomatic characterization another important aspect has been time-consistency of the dynamic risk measures, see e.g. \cite{bion2009time,bielecki2018unified} for the situation without control and \cite{shapiro2012time,shapiro2012minimax} for the situation with control. In the latter reference it is shown that the only time-consistent risk measures are those which iterate  static ones. See also \cite{homem2016risk} for different ways to apply dynamic risk measures. 

Approaches to establish a theory for controlled dynamic risk measures have before been presented in \cite{Ruszcynski2010,shen2013risk,chu2014markov,AsienkiewiczJaskiewicz2017}.  \cite{Ruszcynski2010} is an axiomatic approach. The paper restricts to bounded random variables for the infinite time horizon and uses Markov risk measures to obtain time-consistency. However, some assumptions are indirect properties of the risk measures (see e.g. Theorem 2 in this paper). In \cite{shen2013risk} so-called risk maps are considered and weighted norm spaces are used to treat unbounded rewards. Concepts like sub- and uppermodules are needed to prove the main theorems. \cite{chu2014markov} also treats unbounded cost problems but restricts to coherent risk measures. Moreover, some assumptions on the existence of limits are made because the Fatou property of risk measures is not exploited there. The note \cite{AsienkiewiczJaskiewicz2017} restricts the discussion to entropic risk measures. 

There are also papers which apply recursive risk measures in specific problems. E.g.\ in \cite{jiang2016practicality} a convex combination of expectation and Expected Shortfall is used to tackle the problem of electric vehicle charging in a dynamic decision framework. The authors there compare true risk and expectation with the standard MDP problem. \cite{schur2019time} investigate dynamic pricing problems with dynamic Expected Shortfall and also compare their findings to the standard MDP.  \cite{BaeuerleJaskiewicz2017} consider optimal dividend payments under dynamic entropic risk measures and  \cite{BaeuerleJaskiewicz2018} optimal growth models under dynamic entropic risk measures. In \cite{tamar2016sequential} sampling-based algorithms for coherent risk measures are constructed.

In this paper now we restrict to a recursive application of static risk measures and use unbounded cost functions. The risk measures may be rather general and we state the needed properties for every result. In contrast to the earlier literature our assumptions are in most cases assumptions on the model data alone. We also treat the important case of monotone models where comonotonicity of the risk measures is crucial.  

In more detail the structure of our paper is as follows: In the next section, we summarize some important concepts of risk measures.  We consider in particular distortion risk measures. In Section \ref{sec:decision_model}, we introduce our Markov Decision model. The finite-horizon optimization problem is then considered in Section \ref{sec:finite}. The aim is to minimize recursive risk measures over a finite time horizon. We show here that for proper coherent risk measures with the Fatou property local bounding functions are sufficient for the well-posedness of the optimization problem. Otherwise global bounding function may be necessary. Under some continuity and compactness conditions on the MDP data we show that an optimal policy exists which is Markovian and the value of the problem can be computed recursively. In Section \ref{sec:infinite}, we consider the problem with an infinite time horizon. Here the first result, which states a fixed point property of the value function and the existence of an optimal stationary policy, is under the condition of coherence of the risk measure. In Section \ref{sec:robust}, we briefly discuss the relation to distributionally robust MDP. In Section \ref{sec:monotone}, we consider MDP with monotonicity properties. Here we can work with semicontinuous model data. Another special case arises when the cost function is bounded from below. Then, under the monotonicity assumptions the monetary risk measure does not have to be coherent but  comonotonic additive to obtain the same results. In the last section, we illustrate our results with some examples: We show that in a monotone recursive Value-at-Risk model, the optimal policy is myopic. Moreover, we consider stopping problems, casino games and a cash balance problem where structural properties of the standard MDP formulation still hold.

\section{Risk Measures}\label{sec:risk_measures}

Let a probability space $(\Omega, \A, \P)$ and a real number $p \in [1,\infty)$ be fixed. With $q \in (1,\infty]$ we denote the conjugate index satisfying $\frac1p + \frac1q=1$ under the convention $\frac1\infty=0$. Henceforth, $L^p=L^p(\Omega, \A, \P)$ denotes the vector space of real-valued random variables which have an integrable $p$-th moment. We follow the convention of the actuarial literature that positive realizations of random variables represent losses and negative ones gains. A \emph{risk measure} is a functional $\rho : L^p \to \bar{\R}$. The following properties will be important.

\begin{definition}\label{def:rm_properties}
	A risk measure $\rho : L^p \to \bar{\R}$ is
	\begin{enumerate}
		\item \emph{law-invariant} if $\rho(X)=\rho(Y)$ for $X,Y$ with the same distribution.
		\item \emph{monotone} if $X\leq Y$ implies $\rho(X) \leq \rho(Y)$.
		\item \emph{translation invariant} if $\rho(X+m)=\rho(X)+m$ for all $m \in \R$.
		\item \emph{normalized} if $\rho(0)=0$.
		\item \emph{finite} if $\rho(L^p) \subseteq \R$.
		\item \emph{comonotonic additive} if $\rho(X+Y) = \rho(X)+\rho(Y)$ for all comonotonic $X,Y$.
		\item \emph{positive homogeneous} if $\rho(\lambda X)=\lambda\rho(X)$ for all $\lambda \in \R_+$.
		\item \emph{convex} if $\rho(\lambda X+(1-\lambda)Y)\leq \lambda\rho(X)+(1-\lambda)\rho(Y)$ for $\lambda \in [0,1]$.
		\item \emph{subadditive} if $\rho(X+Y)\leq \rho(X)+\rho(Y)$ for all $X,Y$.
		\item said to have the \emph{Fatou property}, if for every sequence $\{X_n\}_{n \in \N} \subseteq L^p$ with $|X_n| \leq Y$ $\P$-a.s.\ for some $Y \in L^p$ and $X_n \to X$ $\P$-a.s.\ for some $X \in L^p$ it holds
		\[ \liminf_{n \to \infty} \rho(X_n) \geq \rho(X). \]
	\end{enumerate}
\end{definition}

A risk measure is called \emph{monetary} if it is monotone and translation invariant. It appears to be consensus in the literature that these two properties are a necessary minimal requirement for any risk measure. Monetary risk measures which are additionally positive homogeneous and subadditive are referred to as \emph{coherent}. Further, note that positive homogeneity implies normalization and makes convexity and subadditivity equivalent. The Fatou property means that the risk measure is lower semicontinuous w.r.t.\ dominated convergence. 

\begin{lemma}[Theorem 7.24 in \cite{Rueschendorf2013}]\label{thm:finite_convex_fatou}
	Finite and convex monetary risk measures have the Fatou property.
\end{lemma}

Coherent risk measures satisfy a triangular inequality.

\begin{lemma}[Prop. 6 in \cite{Pichler2013}]\label{thm:coherent_triangular}
	For a coherent risk measure $\rho$ and $X,Y \in L^p$ it holds
	\[ \left\lvert \rho(X) - \rho(Y) \right\rvert \leq \rho(|X-Y|). \]
\end{lemma}

We denote by $\M_1(\Omega,\A,\P)$ the set of probability measures on $(\Omega,\A)$ which are absolutely continuous with respect to $\P$ and define
\[ \M_1^q(\Omega,\A,\P) = \left\{ \Q \in\M_1(\Omega,\A,\P): \frac{\dif \Q}{\dif \P} \in L^{q}(\Omega,\A,\P) \right\}. \]
Recall that an extended real-valued convex functional is called \emph{proper} if it never attains $-\infty$ and is strictly smaller than $+\infty$ in at least one point. Coherent risk measures have the following dual or robust representation.

\begin{proposition}[Theorem 7.20 in \cite{Rueschendorf2013}]\label{thm:coherent_risk_measure_dual} 
	A functional $\rho: L^p \to \bar R$ is a proper coherent risk measure with the Fatou property if and only if there exists a subset $\Qc \subseteq \M_1^q(\Omega,\A,\P)$ such that
	\begin{align*}
	\rho(X)= \sup_{\Q \in \Qc} \E^\Q[X], \qquad X \in L^p.
	\end{align*}
	The supremum is attained since the subset $\Qc \subseteq \M_1^q(\Omega,\A,\P)$ can be chosen $\sigma(L^q,L^p)$-compact and the functional $\Q \mapsto \E^\Q[X]$ is $\sigma(L^q,L^p)$-continuous.
\end{proposition}

With the dual representation we can derive a complementary inequality to subadditivity.

\begin{lemma}\label{thm:subadditivity_complement}
	A proper coherent risk measure with the Fatou property $\rho: L^p \to \bar R$ satisfies
	\[ \rho(X+Y) \geq \rho(X) - \rho(-Y) \qquad \text{for all } X,Y \in L^p. \]
\end{lemma}
\begin{proof}
	By Proposition \ref{thm:coherent_risk_measure_dual} it holds for $X,Y \in L^p$
	\begin{align*}
	\rho(X+Y) &= \sup_{\Q \in \Qc} \E^\Q[X+Y]
	= \sup_{\Q \in \Qc} \Big( \E^\Q[X] + \E^\Q[Y] \Big)\\
	&\geq \sup_{\Q \in \Qc}  \E^\Q[X] + \inf_{\Q \in \Qc} \E^\Q[Y] 
	= \rho(X) - \sup_{\Q \in \Qc} \E^\Q[-Y]  \\
	&= \rho(X) - \rho(-Y). \qedhere
	\end{align*}
\end{proof} 

In the following, $F_X(x)=\P(X\leq x)$ denotes the distribution function, $S_X(x)=1-F_X(x), \ x \in \R,$ the survival function and $\q_X(u)=\inf\{x \in \R: F_X(x)\geq u\}, \ u \in [0,1]$, the quantile function of a random variable $X$. Many established risk measures belong to the large class of distortion risk measures.

\begin{definition}\label{def:dist-rm}
	\begin{enumerate}
		\item An increasing function $g:[0,1] \to [0,1]$ with $g(0)=0$ and $g(1)=1$ is called \emph{distortion function}.
		\item The \emph{distortion risk measure} w.r.t.\ a distortion function $g$ is defined by $\rho_g: L^p \to \bar \R$,
		\[\rho_g(X)= \int_0^\infty g(S_X(x))  \dif x  -  \int_{-\infty}^0 1-g(S_X(x)) \dif x \]
		whenever at least one of the integrals is finite.
	\end{enumerate}
\end{definition}

Distortion risk measures have many of the properties introduced in Definition \ref{def:rm_properties}, see e.g.\ \cite{Sereda2010}.
\begin{lemma}\phantomsection \label{thm:dist-rm_properties}
	\begin{enumerate}
		\item Distortion risk measures are law invariant, monotone, translation invariant, normalized, positive homogeneous and comonotonic additive.
		\item A distortion risk measure is subadditive if and only if the distortion function $g$ is concave.
	\end{enumerate}
\end{lemma}

There is an alternative representation of distortion risk measures in terms of Lebesgue-Stieltjes integrals based on the quantile function in lieu of the survival function of the risk $X$. 

\begin{remark}\label{rem:spectral-rm}
	For a distortion risk measure $\rho_g$ with left-continuous distortion function $g$ it holds
	\begin{align}\label{eq:dist-rm_Stieltjes}
	\rho_g(X)= \int_0^1 \q_X(u)\dif \bar g(u),
	\end{align}
	where $\bar g(u)=1-g(1-u), \ u \in [0,1],$ is the dual distortion function, cf.\ \cite{Dhaene2012}. For a continuous concave distortion function $g:[0,1]\to [0,1]$, the dual distortion function $\bar g: [0,1] \to [0,1]$ is continuous convex and can be written as $\bar g(x)= \int_0^x \phi(s) \dif s$ for an increasing right-continuous function $\phi:[0,1] \to \R_+$, which is called \emph{spectrum}. By the properties of the Lebesgue-Stieltjes integral, \eqref{eq:dist-rm_Stieltjes} can then be written as 
	\begin{align}\label{eq:spectral_rm}
	\rho_{g}(X)=\rho_{\phi}(X)= \int_0^1 \q_X(u) \phi(u) \dif u.
	\end{align}
	Therefore, distortion risk measures with continuous concave distortion function are referred to as \emph{spectral risk measures}. Note that continuity of $g$ is an additional requirement only in $0$, since an increasing concave function on $[0,1]$ is already continuous on $(0,1]$.
\end{remark}

Due to Hölder's inequality, spectral risk measures $\rho_\phi:L^p\to \bar R$ with spectrum $\phi \in L^q$ fulfill
\begin{align*}
\left \lvert \rho_{\phi}(X) \right \rvert&=\left \lvert \int_0^1 \q_X(u) \phi(u) \dif u \right \rvert \leq \int_0^1 |\q_X(u)| \phi(u) \dif u 
= \big(\E|\q_X(U)|^p\big)^{\frac1p}   \big(\E|\phi(U)|^q\big)^{\frac1q} < \infty,
\end{align*}
where $U \sim \mathcal{U}([0,1])$ is arbitrary. Hence, they have the Fatou property by Lemma \ref{thm:finite_convex_fatou}. 

\begin{example}
	The most widely used risk measure in finance and insurance \emph{Value-at-Risk}
	\[ \VaR_{\alpha}(X) = \q_X(\alpha), \qquad \alpha \in (0,1), \]
	is a distortion risk measure with distortion function $g(u)=\1_{(1-\alpha,1]}(u)$. Since the distortion function is not concave, Value-at-Risk is not coherent and especially not a spectral risk measure. The lack of coherence can be overcome by using \emph{Expected Shortfall}
	\[ \ES_{\alpha}(X)= \frac{1}{1-\alpha}\int_{\alpha}^1 \q_X(u) \dif u, \qquad \alpha \in [0,1). \]
	The corresponding distortion function $g(u)=	\min\{\frac{u}{1-\alpha},1\}$ is concave and Expected Shortfall thus coherent. It is also spectral with $\phi(u)=\frac{1}{1-\alpha}\1_{[\alpha,1]}(u)$. Due to the bounded spectrum, $\ES$ has the Fatou property. The well-known \emph{entropic risk measure}
	\[ \rho_\gamma(X)= \frac{1}{\gamma} \log \E\left[ e^{\gamma X}  \right], \qquad \gamma >0, \]
	is an example of a law-invariant and convex monetary risk measure which does not belong to the distortion class. For random variables with existing moment-generating function it has the Fatou property directly by dominated convergence. 
\end{example}

To the best of our knowledge, it has surprisingly not been investigated in the literature whether Value-at-Risk has the Fatou property.
\begin{lemma}
	Value-at-Risk has the Fatou property.
\end{lemma}
\begin{proof}
	Assume the contrary. Then there exists a sequence $\{X_n\}_{n \in \N} \subseteq L^p$ with $|X_n| \leq Y$ $\P$-a.s.\ for some $Y \in L^p$ and $X_n \to X$ $\P$-a.s.\ for some $X \in L^p$ such that $$\liminf_{n \to \infty} \VaR_{\alpha}(X_n) < \VaR_{\alpha}(X).$$ I.e.\ there is an $\epsilon >0$ such that for every $\delta \in (0,\epsilon)$
	\[ \liminf_{n \to \infty} \q_{X_n}(\alpha) \leq  \q_X(\alpha) - \delta. \] 
	Hence, there exists a subsequence $\{\q_{X_{N_k}}(\alpha)\}_{k \in \N}$ such that for all $k \in \N$ and $\delta \in (0,\epsilon)$
	\[ \q_{X_{n_k}}(\alpha) \leq  \q_X(\alpha) - \delta   \]
	or equivalently by the properties of generalized inverses $\alpha \leq  F_{X_{n_k}}(\q_X(\alpha) - \delta)$.
	Since $F_X$ has at most countably many discontinuities, we can choose $ \delta_0 \in (0,\epsilon)$ such that $\q_X(\alpha) -  \delta_0$ is a point of continuity of $F_X$. Then, by the definition of convergence in distribution
	\[ \alpha \leq \lim_{k \to \infty} F_{X_{n_k}}(\q_X(\alpha) - \delta_0) = F_{X}(\q_X(\alpha) - \delta_0). \]
	Again  by the properties of generalized inverses, this is equivalent to $\q_X(\alpha) \leq \q_X(\alpha) - \delta_0$, 
	a contradiction.
\end{proof}

\section{The Markov Decision Model}\label{sec:decision_model}

We consider the following standard Markov Decision Process with general Borel state and action spaces. The \emph{state space} $E$ is a Borel space with Borel $\sigma$-algebra $\B(E)$ and the \emph{action space} $A$ is a Borel space with Borel $\sigma$-Algebra $\B(A)$. The possible state-action combinations at time $n$ form a measurable subset $D_n$  of $E \times A$ such that $D_n$ contains the graph of a measurable mapping $E \to A$. The $x$-section of $D_n$,  
\[ D_n(x) = \{ a \in A: (x,a) \in D_n \}, \]
is the set of admissible actions in state $x \in E$ at time $n$. Note that the sets $D_n(x)$ are non-empty. We assume that the dynamics of the MDP are given by measurable \emph{transition functions} $T_n:D_n \times \Z \to E$ and depend on 
\emph{disturbances} $Z_1,Z_2,\dots$ which are independent random elements on a common probability space $(\Omega,\A,\P)$ with values in a measurable space $(\Z, \Zf)$. When the current state is $x_n$, the controller chooses action $a_n\in D_n(x_n)$ and $z_{n+1}$ is the realization of $Z_{n+1}$, then the next state is given by
\[ x_{n+1} = T_n(x_n,a_n,z_{n+1}). \] 
The \emph{one-stage cost function} $c_n:D_n\times E \to \R$ gives the cost $c_n(x,a,x')$  for choosing action $a$ if the system is in state $x$ at time $n$ and the next state is $x'$. The \emph{terminal cost function} $c_N: E \to \R$ gives the cost $c_N(x)$ if the system terminates in state $x$.

The model data is supposed to have the following continuity and compactness properties. 

\begin{assumption}\phantomsection\label{ass:continuity_compactness}
	\begin{enumerate}
		\item[(i)]  The sets $D_n(x)$ are compact and  $E \ni x \mapsto D_n(x)$ are upper semicontinuous, i.e. if $x_k\to x$ and $a_k\in D_n(x_k)$, $k\in\N$, then $(a_k)$ has an accumulation point in $D_n(x)$.
		\item[(ii)] The transition functions $T_n$ are continuous in $(x,a)$.
		\item[(iii)] The one-stage cost functions $c_n$ and the terminal cost function $c_N$ are lower semicontinuous. 
	\end{enumerate}
\end{assumption}

Under a finite planning horizon $N \in \N$, we consider the model data for $n=0,\dots,N-1$. The decision model is called \emph{stationary} if $D,\ T$ do not depend on $n$, the disturbances are identically distributed, the one-stage cost functions are of the form $c_n=\beta^n c$, and the terminal cost function is $\beta^N c_N$, where $\beta \in (0,1]$ is a discount factor. In that case, $Z$ denotes a representative of the disturbance distribution. For a non-stationary model one may think of the discount factor being included in the cost functions. If the model is stationary and the terminal cost is zero, we allow for an \emph{infinite time horizon} $N=\infty$.

For $n \in \N_0$ we denote by $\H_n$ the set of \emph{feasible histories} of the decision process up to time $n$
\begin{align*}
h_n = \begin{cases}
x_0, & \text{if } n=0,\\
(x_0,a_0,x_1, \dots, x_n), & \text{if } n \geq 1,
\end{cases}
\end{align*}
where $a_k \in D_k(x_k)$ for $k \in \N_0$. In order for the controller's decisions to be implementable, they must be based on the information available at the time of decision making, i.e.\ be functions of the history of the decision process. 
\begin{definition}
	\begin{enumerate}
		\item A measurable mapping $d_n: \mathcal{H}_n \to A$ with $d_n(h_n) \in D_n(x_n)$ for every $h_n \in \mathcal{H}_n$ is called  \emph{decision rule} at time $n$. A finite sequence $\pi=(d_0, \dots,d_{N-1})$ is called \emph{$N$-stage policy} and a sequence $\pi=(d_0, d_1, \dots)$ is called \emph{policy}.
		\item A decision rule at time $n$ is called \emph{Markov} if it  depends on the current state only, i.e.\ $d_n(h_n)=d_n(x_n)$ for all $h_n \in \mathcal{H}_n$. If all decision rules are Markov, the ($N$-stage) policy is called \emph{Markov}.
		\item An ($N$-stage) policy $\pi$ is called \emph{stationary} if $\pi=(d, \dots,d)$ or $\pi=(d,d,\dots)$, respectively, for some Markov decision rule $d$.
	\end{enumerate}
\end{definition}
With $\Pi \supseteq \Pi^M \supseteq \Pi^S$ we denote the sets of all policies, Markov policies and stationary policies. It will be clear from the context if $N$-stage or infinite stage policies are meant. An admissible policy always exists as $D$ contains the graph of a measurable mapping.

Since risk measures are defined as real-valued mappings of random variables, we will work with a functional representation of the decision process. The law of motion does not need to be specified explicitly. We define for an initial state $x_0 \in E$ and a policy $\pi \in \Pi$
\begin{align*}
X^\pi_0=x_0, \qquad X^\pi_{n+1}= T(X_n^\pi,d_n(H_n^\pi),Z_{n+1}).
\end{align*}
Here, the process $(H_n^\pi)_{n \in \N_0}$ denotes the history of the decision process viewed as a random element, i.e.
\begin{align*}
H_0^\pi=x_0, \quad
H_1^\pi=\big(X_0^\pi,d_0(X_0^\pi),X_1^\pi\big),\quad \dots, \quad H_{n}^\pi=(H_{n-1}^\pi,d_{n-1}(H_{n-1}^\pi),X_{n}^\pi).
\end{align*}
Under a Markov policy the recourse on the random history of the decision process is not needed.

\section{Cost Minimization under a Finite Planning Horizon}\label{sec:finite}

For a finite planning horizon $N \in \N$, we consider the non-stationary decision model. In the classical context of the risk-neutral expected cost criterion, the value of a policy $\pi \in \Pi$ at time $n=0,\dots,N$ given $h_n \in \H_n$ is defined as
\[ V_{n\pi}(h_n)= \E_{nh_n}\left[ \sum_{k=n}^{N-1} c_k(X_k^\pi,d_k(H_k^\pi),X_{k+1}^\pi) + c_N(X_N^\pi) \right], \qquad n=0,\dots,N, \]
where $\E_{nh_n}$ is the conditional expectation given $H_n^\pi=h_n$. Under suitable integrability conditions, $V_{n\pi}$ satisfies the value iteration
\[ V_{n\pi}(h_n)= \E\Big[ c_n\big(x_n,d_n(h_n),T_n(x_n,d_n(h_n),Z_{n+1})\big) + V_{n+1\pi}\big(h_n,d_n(h_n),T_n(x_n,d_n(h_n),Z_{n+1})\big) \Big], \] 
see e.g.\ Theorem 2.3.4 in \cite{BaeuerleRieder2011}. In order to take risk-sensitive preferences of the controller into account, the approach here is to replace the factorization of conditional expectation in the value iteration by a risk measure, meaning that static risk measures are recursively applied at each stage. In the special case of the entropic risk measure, this approach has been studied by  \cite{AsienkiewiczJaskiewicz2017} in an abstract setting and by \cite{BaeuerleJaskiewicz2017,BaeuerleJaskiewicz2018} in applications to optimal dividend payments and stochastic optimal growth. Their choice of the risk measure is motivated by the fact that the entropic risk measure coincides with the certainty equivalent of an exponential utility function. In the economic literature, recursive utilities have been widely studied. For a literature overview we refer the reader to  \cite{Miao2014}. 

Let $p \in [1,\infty)$ with conjugate index $q \in [1,\infty]$ and let $\rho_0,\dots,\rho_{N-1}: L^p(\Omega,\A,\P) \to \bar \R$ be monetary risk measures. We define the \emph{value of a policy} $\pi = (d_0,\dots,d_{N-1}) \in \Pi$ at time $n = 0,\dots, N$ given history $h_n \in \mathcal{H}_n$ recursively as 
\begin{align*}
V_{N\pi}(h_N) &=c_N(x_N), \\
V_{n\pi}(h_n)&= \rho_n\Big( c_n\big(x_n,d_n(h_n),T_n(x_n,d_n(h_n),Z_{n+1})\big) \! + \!  V_{n+1\pi}\big(h_n,d_n(h_n),T_n(x_n,d_n(h_n),Z_{n+1})\big) \Big). 
\end{align*}
In the special case that the one-stage cost functions $c_n$ do not depend on the next state of the decision process, the value of a policy simplifies to
\[ V_{n\pi}(h_n) = c_n(x_n,d_n(h_n)) + \rho_n\big( V_{n+1\pi}(h_n,d_n(h_n),X_{n+1}^\pi) \big), \qquad h_n \in \H_n, \]
due to the translation invariance of monetary risk measures. 

\begin{remark}\label{rem:measurability}
	For the recursive definition of the policy values to be meaningful, we need to make sure that the risk measures are applied to elements of $L^p(\Omega,\A,\P)$. This has two aspects: integrability will be ensured by Assumption \ref{ass:finite}, but first of all $V_{n\pi}$ needs to be a measurable function for all $\pi \in \Pi$ and $n=0,\dots,N$. For most risk measures with practical relevance, this is fulfilled: 
	\begin{itemize}
		\item In the risk-neutral case, i.e.\ for $\rho=\E$, and also for the entropic risk measure $\rho_\gamma$ the measurability is obvious.
		\item For distortion risk measures, the measurability is guaranteed, too. To see this, we proceed backwards. For $N$ there is noting to show and if $V_{n+1\pi}$ is measurable, the function
		\[ f(h_n,z) = c_n\big(x_n,d_n(h_n),T_n(x_n,d_n(h_n),z)\big) \! + \!  V_{n+1\pi}\big(h_n,d_n(h_n),T_n(x_n,d_n(h_n),z) \]
		is measurable as a composition of measurable maps. Then, Fubini's theorem yields that the survival function of $f(h_n,Z_{n+1})$
		\[ S(t|h_n) = \int \1\{ f(h_n,Z_{n+1}(\omega)) > t \} \P(\dif \omega) \] 
		is measurable. A distortion function $g$ is increasing and hence measurable. So again by Fubini's theorem we obtain the measurability of 
		\begin{align*}
			V_{n\pi}(h_n)= \rho_g(f(h_n,Z_{n+1})) = \int_0^\infty g(S(t|h_n))  \dif t  -  \int_{-\infty}^0 1-g(S(t|h_n)) \dif t
		\end{align*}
		since the integrands are non-negative and compositions of measurable maps. 
		\item For proper coherent risk measures with the Fatou property one can insert the dual representation of Proposition \ref{thm:coherent_risk_measure_dual}. Then, an optimal measurable selection argument as in Theorem 3.6 in \cite{BauerleGlauner2020} yields the measurability.
	\end{itemize}
	Throughout, it is implicitly assumed that the risk measures are chosen such that all policy values are measurable.
\end{remark}

The \emph{value functions} are given by 
\[ V_n(h_n) = \inf_{\pi \in \Pi}  \ V_{n\pi}(h_n), \qquad h_n \in \mathcal{H}_n, \]
for $n=0,\dots,N$ and the controller's optimization objective is 
\begin{align*}
V_0(x)=\inf_{\pi \in \Pi} \ V_{0\pi}(x), \qquad x \in E.
\end{align*}

In order to have well-defined value functions, we need some finiteness conditions instead of the usual integrability conditions. Moreover, we require some basic properties for the risk measures.

\begin{assumption}\phantomsection\label{ass:finite}
	\begin{enumerate}
		\item[(i)] There exist $\ubar \epsilon, \bar \epsilon \geq 0$ with $\ubar \epsilon+ \bar \epsilon =1$ and measurable functions $\ubar \b: E \to (-\infty, - \ubar \epsilon]$ and $\bar \b: E \to [\bar \epsilon, \infty)$ such that it holds for all policies $\pi \in \Pi$ and all $n=0,\dots,N$
		\[ \ubar \b(x_n) \leq V_{n\pi}(h_n) \leq \bar \b(x_n), \qquad h_n \in \mathcal{H}_n. \]
		\item[(ii)] We define $\b:E \to [1,\infty), \ \b(x)=\bar \b(x) - \ubar \b(x)$. For all $n=0,\dots,N-1$ and $(\bar x, \bar a) \in D_n$ there exists an $\epsilon >0$ and measurable functions $\Theta_{n,1}^{\bar x, \bar a}, \Theta_{n,2}^{\bar x, \bar a}: \Z \to \R_+$ such that $\Theta_{n,1}^{\bar x, \bar a}(Z_{n+1}), \Theta_{n,2}^{\bar x, \bar a}(Z_{n+1}) \in L^p(\Omega,\A,\P)$ and
		\begin{align*}
		|c_n(x,a,T_n(x,a,z))| \leq \Theta_{n,1}^{\bar x, \bar a}(z), \qquad \qquad  \b(T_n(x,a,z)) \leq \Theta_{n,2}^{\bar x, \bar a}(z)
		\end{align*}
		for all  $z \in \Z$ and $(x,a) \in B_\epsilon(\bar x, \bar a) \cap D_n$. Here, $B_\epsilon(\bar x, \bar a)$ is the closed ball around $(\bar x, \bar a)$ w.r.t.\ an arbitrary product metric on $E\times A$.
		\item[(iii)] The monetary risk measures $\rho_0,\dots,\rho_{N-1}: L^p(\Omega,\A,\P) \to \bar \R$ are law invariant and have the Fatou property.
	\end{enumerate}
\end{assumption}

$\ubar \b, \bar \b$ are called \emph{(global) lower} and \emph{upper bounding function}, respectively, while $\b$ is referred to as \emph{(global) bounding function}. Since  $\ubar \b$ is non-positive and $\bar \b$ is non-negative it holds
\[ \ubar \b(x_n) \leq - V_{n\pi}^-(h_n) \leq V_{n\pi}(h_n) \leq V_{n\pi}^+(h_n) \leq \bar \b(x_n), \qquad h_n \in \mathcal{H}_n, \]
and consequently $|V_{n\pi}(h_n)| \leq \b(x_n)$. Bold print is used to distinguish these global bounding functions from the usual local (stage-wise) bounding functions used for risk-neutral MDP. Such local bounding functions can be introduced for the risk-sensitive recursive optimality criterion, too, if the risk measures have additional properties. Note that without any further properties on the risk measure we cannot construct global bounding functions from local ones.

\begin{lemma}\label{thm:stagewise_bound}
	Let $\rho_0,\dots,\rho_{N-1}$ be proper coherent risk measures with the Fatou property. If there exist  $\ubar \epsilon, \bar \epsilon \geq 0$ with $\ubar \epsilon+ \bar \epsilon =1$, measurable functions $\ubar b:E \to (-\infty, -\ubar \epsilon]$, $\bar b: E \to [\bar \epsilon,\infty)$ and a constant $\alpha \in (0,1)$ such that
	\begin{align*}
	\rho_n\big( c_n(x,a,T_n(x,a,Z_{n+1}))\big)& \geq \ubar b(x), & \rho_n\big( - \ubar b(T_n(x,a,Z_{n+1}))\big) &\leq - \alpha \ubar b(x),\\
	\rho_n\big( c_n(x,a,T_n(x,a,Z_{n+1}))\big) &\leq \bar b(x), & \rho_n\big( \bar b(T_n(x,a,Z_{n+1}))\big) &\leq \alpha \bar b(x),
	\end{align*}
	for all $n=0,\dots,N-1$ and $(x,a) \in D_n$ as well as $\ubar b(x) \leq c_N(x) \leq \bar b(x)$ for all $x \in E$, then
	\begin{align*}
	\ubar \b = \frac{1}{1-\alpha}\ubar b, \qquad \bar \b = \frac{1}{1-\alpha}\bar b \qquad \text{and} \qquad \b = \frac{1}{1-\alpha} b
	\end{align*}
	are global bounding functions satisfying Assumption \ref{ass:finite} (i).
\end{lemma}
\begin{proof}
	We proceed by backward induction. At time $N$ we have 
	\begin{align*}
	\ubar \b(x_N) \leq \ubar b(x_N) \leq c_N(x_N) \leq \bar b(x_N) \leq \bar \b(x_N), \qquad h_N \in \H_N.
	\end{align*}
	Assuming the assertion holds for time $n+1$ it follows for time $n$:
	\begin{align*}
	V_{n\pi}(h_n) &= \rho_n\Big( c_n\big(x_n,d_n(h_n),T_n(x_n,d_n(h_n),Z_{n+1})\big)+ V_{n+1\pi}\big(h_n,d_n(h_n),T_n(x_n,d_n(h_n),Z_{n+1})\big) \Big)\\
	&\geq \rho_n\Big( c_n\big(x_n,d_n(h_n),T_n(x_n,d_n(h_n),Z_{n+1})\big) + \frac{1}{1-\alpha} \ubar b\big(T_n(x_n,d_n(h_n),Z_{n+1})\big) \Big)\\
	&\geq \rho_n\Big( c_n\big(x_n,d_n(h_n),T_n(x_n,d_n(h_n),Z_{n+1})\big)\Big) - \frac{1}{1-\alpha} \rho_n \Big( - \ubar b\big(T_n(x_n,d_n(h_n),Z_{n+1})\big) \Big)\\
	&\geq \ubar b(x_n) + \frac{\alpha}{1-\alpha} \ubar b(x_n) = \ubar \b(x_n).
	\end{align*}
	The second inequality is by Lemma \ref{thm:subadditivity_complement}. Regarding the upper bounding function one can argue similarly using the subadditivity of $\rho_n$ instead. 
	\begin{align*}
	V_{n\pi}(h_n) &= \rho_n\Big( c_n\big(x_n,d_n(h_n),T_n(x_n,d_n(h_n),Z_{n+1})\big) + V_{n+1\pi}\big(h_n,d_n(h_n),T_n(x_n,d_n(h_n),Z_{n+1})\big) \Big)\\
	&\leq \rho_n\Big( c_n\big(x_n,d_n(h_n),T_n(x_n,d_n(h_n),Z_{n+1})\big) + \frac{1}{1-\alpha} \bar b\big(T_n(x_n,d_n(h_n),Z_{n+1})\big) \Big)\\
	&\leq \rho_n\Big( c_n\big(x_n,d_n(h_n),T_n(x_n,d_n(h_n),Z_{n+1})\big)\Big) + \frac{1}{1-\alpha} \rho_n \Big(\bar b\big(T_n(x_n,d_n(h_n),Z_{n+1})\big) \Big)\\
	&\leq \bar b(x_n) + \frac{\alpha}{1-\alpha} \bar b(x_n) = \bar \b(x_n). \qedhere
	\end{align*} 	
\end{proof}

\begin{remark}\phantomsection\label{rem:stagewise_bound}
	\begin{enumerate}
		\item Concerning the requirements on a local lower bounding function in Lemma \ref{thm:stagewise_bound} it should be noted that $\rho_n\big( - \ubar b(T_n(x,a,Z_{n+1}))\big) \leq - \alpha \ubar b(x)$ is a stronger assumption than
		\begin{align}\label{eq:stagewise_bound_weaker}
		\rho_n\big(\ubar b(T_n(x,a,Z_{n+1}))\big) \geq \alpha \ubar b(x).
		\end{align}  
		Indeed, since $\ubar b \leq 0$ the monotonicity and normalization of $\rho_n$ yields $\rho_n\big(\ubar b(T_n(x,a,Z_{n+1}))\big) \leq 0$. Consequently, we have by Lemma \ref{thm:coherent_triangular}
		\begin{align*}
		-\rho_n\big(\ubar b(T_n(x,a,Z_{n+1}))\big) &= \left\lvert \rho_n\big(\ubar b(T_n(x,a,Z_{n+1}))\big) \right\rvert
		\leq \rho_n\Big(\left\lvert \ubar b \big( T_n(x,a,Z_{n+1}) \big) \right\rvert\Big)\\
		&= \rho_n\Big(- \ubar b \big( T_n(x,a,Z_{n+1}) \big) \Big) \leq -\alpha \ubar b(x).
		\end{align*}
		Multiplying with $(-1)$ yields \eqref{eq:stagewise_bound_weaker}.
		\item If the one-stage cost functions are bounded and the monetary risk measures $\rho_0,\dots, \rho_{N-1}$ normalized, the local bounding functions $\ubar b, \bar b$ can be chosen constant. Where we have used Lemma \ref{thm:subadditivity_complement} or subadditivity in the proof of Lemma \ref{thm:stagewise_bound}, one can then simply argue with translation invariance. Note that normalization is no structural restriction for monetary risk measures due to the translation invariance.
	\end{enumerate}
\end{remark}

With the bounding function $\b$ we define the function space
\[ \BB_\b = \left\{ v: E \to \R \ \vert \ v \text{ measurable with } \lambda \in \R_+ \text{ s.t. } |v(x)| \leq \lambda \,  \b(x) \text{ for all } x \in E \right\}. \]
Endowing $\BB_\b$ with the weighted supremum norm
\[ \|v\|_\b = \sup_{x \in E} \frac{|v(x)|}{\b(x)} \]
makes $(\BB_\b, \|\cdot\|_\b )$ a Banach space, cf. Proposition 7.2.1 in \cite{HernandezLasserre1999}. In case we have local bounding functions as in Lemma \ref{thm:stagewise_bound}, it holds
\begin{align*}
\BB_\b &= \left\{ v: E \to \R \ \vert \ v \text{ measurable with } \lambda \in \R_+ \text{ s.t. } |v(x)| \leq \lambda \,  \b(x) \text{ for all } x \in E \right\}\\
&= \left\{ v: E \to \R \ \vert \ v \text{ measurable with } \lambda \in \R_+ \text{ s.t. } |v(x)| \leq \lambda \,  b(x) \text{ for all } x \in E \right\}\\
&= \BB_b
\end{align*} 
and the weighted supremum norms $\|\cdot \|_\b, \ \|\cdot \|_b$ are equivalent. 

\begin{lemma}\label{thm:Lpbound}
	Let $v \in \BB_\b$ and $n \in \{0,\dots,N-1\}$. Under Assumptions \ref{ass:continuity_compactness} (i) and  \ref{ass:finite} (ii) each sequence of random variables
	\[ C_k= c_n\big(x_k,a_k,T_n(x_k,a_k,Z_{n+1})\big) + v\big(T_n(x_k,a_k,Z_{n+1})\big) \] 
	induced by a convergent sequence $\{(x_k,a_k)\}_{k \in \N}$ in $D_n$ has an $L^p$-bound $\bar C$, i.e.\ $|C_k| \leq \bar C \in L^p(\Omega,\A,\P)$ for all $k \in \N$. 
\end{lemma}
\begin{proof}
	There exists a constant $\lambda \in \R_+$ such that $|v|\leq \lambda b$. Since $D_n$ is closed by Lemma A.2.2 in \cite{BaeuerleRieder2011}, the limit point $(x_0,a_0)$ of $\{(x_k,a_k)\}_{k \in \N}$ lies in $D_n$. Let $\epsilon>0$ be the constant from Assumption \ref{ass:finite} (ii) corresponding to $(x_0,a_0)$. Since the sequence is convergent, there exists $m \in \N$ such that $(x_k,a_k) \in B_\epsilon(x_0,a_0) \cap D_n$ for all $k > m$. For the finite number of points $(x_0,a_0),  (x_1,a_1), \dots, (x_m,a_m)$ there exist bounding functions $\Theta_{n,1}^{x_i,a_i}, \Theta_{n,2}^{x_i,a_i}$ by Assumption \ref{ass:finite} (ii). Thus, the random variable 
	\[ \bar C= \max_{i=0,\dots,m} \Big(\Theta_{n,1}^{x_i,a_i}(Z) + \lambda \Theta_{n,2}^{x_i,a_i}(Z)  \Big) \]
	is an $L^p$-bound as desired.
\end{proof}

Let us now consider specifically Markov policies $\pi \in \Pi^M$ of the controller. The subspace 
\[ \BB = \{ v \in \BB_\b : \ v \text{ lower semicontinuous}  \} \]
of $(\BB_\b, \|\cdot\|_\b )$ turns out to be the set of potential value functions under such policies. $(\BB, \|\cdot\|_\b )$ is a complete metric space since the subset of lower semicontinuous functions is closed in $(\BB_\b, \|\cdot\|_\b )$. When we consider intervals $[\ubar v, \bar v] \subseteq \BB$ with $\ubar v,\bar v: E \to \R$ s.t.\ $ \ubar v(x) \leq \bar v(x)$ for all $x \in E$, they are to be understood pointwise
\[ [\ubar v, \bar v] = \{ v \in \BB: \ubar v(x) \leq v(x) \leq \bar v(x) \text{ for all } x \in E \}. \]
Such intervals are closed  even w.r.t.\ pointwise convergence and therefore form a complete metric space as a closed subset of $(\BB, \|\cdot\|_\b )$. In the sequel, the interval
\[ I= \left[ \ubar \b, \bar \b\right] \]
will be of interest. We define the following operators on $\BB_\b$ and especially on $\BB$. 

\begin{definition}\label{def:operators}
	For $v \in \BB_b$ and a Markov decision rule $d$ let
	\begin{align*}
	L_n v (x,a) &= \rho_n\Big( c_n\big(x,a,T_n(x,a,Z_{n+1})\big) + v\big(T_n(x,a,Z_{n+1})\big) \Big),  && (x,a) \in D_n,\\
	\T_{n d} v(x) &= L_n v(x,d(x)), && x \in E,\\
	\T_n  v(x) &= \inf_{a \in D_n(x)}  L_n v (x,a), && x \in E.
	\end{align*}
\end{definition}

Note that the operators are monotone in $v$. Under a Markov policy $\pi=(d_0, \dots, d_{N-1}) \in \Pi^M$, the value iteration can be expressed with the operators. In order to distinguish from the history-dependent case, we denote policy values here with $J$. Setting $J_{N\pi}(x) = c_N(x), \ x \in E$, we obtain for $n=0, \dots, N-1$ and $x \in E$
\begin{align*}
J_{n\pi}(x) &= \rho_n\Big( c_n\big(x,d_n(x),T_n(x,d_n(x),Z_{n+1})\big) +  J_{n+1\pi}\big(T_n(x,d_n(x),Z_{n+1})\big) \Big) = \T_{n d_n} J_{n+1\pi}(x).
\end{align*} 
Let us further define for $n=0, \dots, N-1$ the Markov value function 
\begin{align*}
J_{n}(x) = \inf_{\pi \in \Pi^M} J_{n\pi}(x) , \qquad x \in E.
\end{align*}

The next result shows that $V_n$ satisfies a Bellman equation and proves that an optimal policy exists and is Markov.
\begin{theorem}\label{thm:finite}
	Let Assumptions \ref{ass:continuity_compactness} and \ref{ass:finite} be satisfied. Then, for $n=0, \dots, N$, the value function $V_n$ only depends on $x_n$, i.e.\ $V_n(h_n)=J_n(x_n)$ for all $h_n \in \H_n$, lies in $I=\left[\ubar \b, \bar \b  \right] \subseteq \BB$ and satisfies the Bellman equation
	\begin{align*}
	J_N(x) &= c_N(x),\\
	J_n(x) &= \T_n J_{n+1}(x), \qquad x \in E.
	\end{align*}
	Furthermore, for $n= 0, \dots, N-1$ there exist Markov decision rules $d_n^*$ such that $\T_n J_{n+1}=\T_{n d_n^*} J_{n+1}$ and every sequence of such minimizers constitutes an optimal policy $\pi=(d_0^*,\dots,d_{N-1}^*)$.
\end{theorem}
\begin{proof}
	The proof is by backward induction. At time $N$ we have $V_N=J_N=c_N$ which is in $\BB$ by Assumptions \ref{ass:continuity_compactness} (iii) and \ref{ass:finite} (i). Assuming the assertion holds at time $n+1$, we obtain for time $n$:
	\begin{align}
	V_n(h_n) &= \inf_{\pi \in \Pi} V_{n\pi}(h_n)\notag\\
	&= \inf_{\pi \in \Pi} \rho_n\Big( c_n\big(x_n,d_n(h_n),T_n(x_n,d_n(h_n),Z_{n+1})\big) \! + \!  V_{n+1\pi}\big(h_n,d_n(h_n),T_n(x_n,d_n(h_n),Z_{n+1})\big) \Big)\notag\\
	&\geq \inf_{\pi \in \Pi} \rho_n\Big( c_n\big(x_n,d_n(h_n),T_n(x_n,d_n(h_n),Z_{n+1})\big) \! + \!  V_{n+1}\big(h_n,d_n(h_n),T_n(x_n,d_n(h_n),Z_{n+1})\big) \Big)\notag\\
	&= \inf_{\pi \in \Pi} \rho_n\Big( c_n\big(x_n,d_n(h_n),T_n(x_n,d_n(h_n),Z_{n+1})\big) + J_{n+1}\big(T_n(x_n,d_n(h_n),Z_{n+1})\big) \Big).\notag\\
	&= \inf_{a_n \in D(x_n)} \rho_n\Big( c_n\big(x_n,a_n,T_n(x_n,a_n,Z_{n+1})\big) +  J_{n+1}\big(T_n(x_n,a_n,Z_{n+1})\big) \Big).\label{eq:finite_proof_0}
	\end{align}
	The last equality holds since the minimization does not depend on the entire policy but only on $a_n=d_n(h_n)$. Here, objective and constraint depend on the history of the process only through $x_n$. Thus, given existence of a minimizing Markov decision rule $d_n^*$, \eqref{eq:finite_proof_0} equals $\T_{n d_n^*} J_{n+1}(x_n)$. Again by the induction hypothesis there exists an optimal Markov policy $\pi^* \in \Pi^M$ such that $J_{n+1}=J_{n+1\pi^*}$. Hence, we have
	\begin{align*}
		V_n(h_n) \geq \T_{n d_n^*} J_{n+1}(x_n) = \T_{n d_n^*} J_{n+1\pi^*}(x_n) = J_{n\pi^*}(x_n) \geq J_{n}(x_n) \geq V_n(h_n).
	\end{align*}
	It remains to show the existence of a minimizing Markov decision rule $d_n^*$ and that $J_n \in \BB$. We want to apply Proposition 2.4.3 in \cite{BaeuerleRieder2011}. The set-valued mapping $E \ni x\mapsto D_n(x)$ is compact-valued and upper semicontinuous. Next, we show that $D_n \ni (x,a) \mapsto  L_n v (x,a)$ is lower semicontinuous for every $v \in \BB$. Let $\{(x_k,a_k)\}_{k \in \N}$ be a convergent sequence in $D_n$ with limit $(x^*,a^*) \in D_n$. The function $D_n \ni (x,a) \mapsto c_n\big(x,a,T_n(x,a,Z_{n+1}(\omega))\big)+ v\big(T_n(x,a,Z_{n+1}(\omega))\big)$ is lower semicontinuous for every $\omega \in \Omega$ as a composition of a continuous and a lower semicontinuous one. Consequently,
	\begin{align}
	&\lim_{k \to \infty} \inf_{\ell \geq k} c_n\big(x_\ell,a_\ell,T_n(x_\ell,a_\ell,Z_{n+1})\big) + v\big(T_n(x_\ell,a_\ell,Z_{n+1})\big) \notag\\
	&= \liminf_{k \to \infty} c_n\big(x_k,a_k,T_n(x_k,a_k,Z_{n+1})\big) + v\big(T_n(x_k,a_k,Z_{n+1})\big) \notag\\
	&\geq c_n\big(x^*,a^*,T_n(x^*,a^*,Z_{n+1})\big) + v\big(T_n(x^*,a^*,Z_{n+1})\big). \label{eq:finite_proof_2}
	\end{align} 
	The sequence $\{C_k\}_{k \in \N}$ with 
	$$C_k(\omega)= \inf_{\ell \geq k} c_n\big(x_\ell,a_\ell,T(x_\ell,a_\ell,Z_{n+1})\big) + v\big(T_n(x_\ell,a_\ell,Z_{n+1})\big)$$
	is measurable as the $\omega$-wise infimum of a countable number of random variables and increasing for every $\omega \in \Omega$. By Lemma \ref{thm:Lpbound}, there exists a nonnegative random variable $\bar C \in L^p(\Omega,\A,\P)$ such that $|C_k| \leq \bar C$ for all $k \in \N$. Hence, $\{C_k\}_{k \in \N}$ converges almost surely to some $C^* \in L^p(\Omega,\A,\P)$. The Fatou property of the risk measure $\rho_n$ implies
	\begin{align*}
	\liminf_{k \to \infty} L_n v(x_k,a_k)	&= \liminf_{k \to \infty} \rho_n\Big( c_n\big(x_k,a_k,T_n(x_k,a_k,Z_{n+1})\big) +  v\big(T_n(x_k,a_k,Z_{n+1})\big) \Big)\\
	& \geq \liminf_{k \to \infty} \rho_n(C_k)\\
	& \geq \rho_n(C^*)\\
	& \geq \rho_n\Big( c_n\big(x^*,a^*,T_n(x^*,a^*,Z_{n+1})\big) + v\big(T_n(x^*,a^*,Z_{n+1})\big) \Big)\\
	& = L_n v(x^*,a^*). 
	\end{align*}
	The last inequality follows from \eqref{eq:finite_proof_2} and the monotonicity of $\rho_n$. So we have shown the lower semicontinuity of $D_n \ni (x,a) \mapsto  L_n v (x,a)$. Proposition 2.4.3 in \cite{BaeuerleRieder2011} yields the existence of a minimizing Markov decision rule $d_n^*$ and that $J_n=T J_{n+1}$ is lower semicontinuous. Furthermore, $J_n$ is bounded by $\ubar \b$ and $\bar \b$ according to Assumption \ref{ass:finite} (i). Thus, $J_n \in I$ and the proof is complete.
\end{proof}

\section{Cost Minimization under an Infinite Planning Horizon}\label{sec:infinite}

In this section, we consider the risk-sensitive recursive cost minimization problem with an infinite planning horizon. This is reasonable if the terminal period is unknown or if one wants to approximate a model with a large but finite planning horizon. Solving the infinite horizon problem will turn out to be easier since it admits a stationary optimal policy. We study the stationary version of the decision model with no terminal cost. Therefore, the risk measure may no longer vary over time. We  also require coherence as an additional property. Recall that if $\rho$ is finite on $L^p(\Omega,\A,\P)$, the Fatou property is already implied by coherence. Within the class of distortion risk measures requiring coherence essentially means a restriction to spectral risk measures. For spectral risk measures, finiteness is guaranteed if the spectrum $\phi$ lies in $L^q$. Due to coherence we can work with local bounding functions, see Lemma \ref{thm:stagewise_bound}. We will see that if the one-stage cost function is bounded, coherence can be dropped as a requirement on the risk measure. Then, all distortion risk measures with the Fatou property are admissible. For clarity, all assumptions of this section are summarized below.

\begin{assumption}\phantomsection \label{ass:infinite}
	\begin{enumerate}
		\item[(i)] There exist  $\alpha, \ubar \epsilon, \bar \epsilon \geq 0$ with $\ubar \epsilon+ \bar \epsilon =1$ and  measurable functions $\ubar b:E \to (-\infty, -\ubar \epsilon]$, $\bar b:E \to [\bar \epsilon,\infty)$ such that for all $(x,a) \in D$
		\begin{align*}
			\rho\big( c(x,a,T(x,a,Z))\big)& \geq \ubar b(x), & \rho\big( -\ubar b(T(x,a,Z))\big) &\leq -\alpha \ubar b(x),\\
			\rho\big( c(x,a,T(x,a,Z))\big) &\leq \bar b(x), & \rho\big( \bar b(T(x,a,Z))\big) &\leq \alpha \bar b(x).
		\end{align*}
		\item[(ii)] We define $b:E \to [1,\infty), \ b(x)=\bar b(x) - \ubar b(x)$. For all $(\bar x, \bar a) \in D$ there exists an $\epsilon >0$ and measurable functions $\Theta_{1}^{\bar x, \bar a}, \Theta_{2}^{\bar x, \bar a}: \Z \to \R_+$ such that $\Theta_{1}^{\bar x, \bar a}(Z), \Theta_{2}^{\bar x, \bar a}(Z) \in L^p(\Omega,\A,\P)$ and
		\begin{align*}
			|c(x,a,T(x,a,z))| \leq \Theta_{1}^{\bar x, \bar a}(z), \qquad \qquad  b(T(x,a,z)) \leq \Theta_{2}^{\bar x, \bar a}(z)
		\end{align*}
		for all  $z \in \Z$ and $(x,a) \in B_\epsilon(\bar x, \bar a) \cap D$. Here, $B_\epsilon(\bar x, \bar a)$ is the closed ball around $(\bar x, \bar a)$ w.r.t.\ an arbitrary product metric on $E \times A$.
		\item[(iii)] The law-invariant risk measure $\rho: L^p(\Omega,\A,\P) \to \bar \R$ is proper, coherent and has the Fatou property.
		\item[(iv)] The discount factor $\beta$ satisfies $\alpha\beta <1$.
	\end{enumerate}
\end{assumption}

Due to  discounting, the global bounding functions corresponding to $\ubar b,\bar b, b$ are given by
\begin{align}\label{eq:infinite_global_bounding}
	\ubar \b = \frac{1}{1-\alpha\beta}\ubar b, \qquad \bar \b = \frac{1}{1-\alpha\beta}\bar b \qquad \text{and} \qquad \b = \frac{1}{1-\alpha\beta} b.
\end{align}
This can be seen as in the proof of Lemma  \ref{thm:stagewise_bound}. 

Since the model with infinite planning horizon will be derived as a limit of the one with finite horizon, the consideration can be restricted to Markov policies $\pi=(d_1,d_2,\dots) \in \Pi^M$ due to Theorem \ref{thm:finite}. When calculating limits, it is convenient to index the value functions with the distance to the time horizon rather than the point in time. This is also referred to as \emph{forward form} of the value iteration. 

\begin{definition}\label{def:operators2}
	For $v \in \BB_b$ and a Markov decision rule $d$ let
	\begin{align*}
	\T_{d} v(x) &=  \rho\Big( c\big(x,d(x),T(x,d(x),Z)\big) + \beta v\big(T(x,d(x),Z)\big) \Big), && x \in E,\\
	\T  v(x) &= \inf_{a \in D(x)}  \rho\Big( c\big(x,a,T(x,a,Z)\big) + \beta v\big(T(x,a,Z)\big) \Big),   && x \in E.
	\end{align*}
\end{definition}

The value of a policy $\pi=(d_0, d_1\dots ) \in \Pi^M$ up to a planning horizon $N \in \N$ now is
\begin{align}\label{eq:infinite_finite_policy_value}
	J_{N\pi}(x) & = \T_{d_0} \circ \dots \circ \T_{d_{N-1}} 0(x), \qquad x \in E.
\end{align}
In a non-stationary formulation the discounting is included in the one-stage cost functions and therefore calibrated w.r.t.\ the fixed reference time zero. If the value functions are considered at a later point in time, the non-stationary and stationary version differ by a discounting factor:
\[ J_{n}^{\text{non-stat}}(x)  = \beta^n J_{N-n}^{\text{stat}}(x), \qquad x \in E, \ n=0,\dots,N. \]

The reformulation \eqref{eq:infinite_finite_policy_value} makes it necessary to write the value iteration in terms of the \emph{shifted policy} $\vec \pi= (d_1,d_2, \dots)$ corresponding to $\pi=(d_0,d_1,\dots) \in \Pi^M$:
\begin{align*}
	J_{N\pi}(x) &= \T_{d_0} J_{N-1 \vec \pi}(x) = \rho\Big( c\big(x,d_0(x),T(x,d_0(x),Z)\big) + \beta J_{N-1 \vec \pi}\big(T(x,d_0(x),Z)\big) \Big), \qquad x \in E.
\end{align*}
The value function under planning horizon $N \in \N$ is given by
\begin{align*}
	J_N(x) = \inf_{\pi \in \Pi^M} J_{N\pi}(x), \qquad x \in E,
\end{align*}
By Theorem \ref{thm:finite}, the value function satisfies the Bellman equation
\begin{align}\label{eq:infinite_finite_Bellman}
	J_N(x) =\T J_{N-1}(x) = \T^N 0(x), \qquad x \in E.
\end{align}
When the planning horizon is infinite, we define the value of a policy $\pi \in \Pi^M$ as 
\begin{align} \label{eq:infinite_policy_value}
	J_{\infty \pi}(x)= \lim_{N \to \infty} J_{N\pi}(x), \qquad x \in E. 
\end{align}
Hence, the optimality criterion considered in this section is
\begin{align}\label{eq:opt_crit_infinite}
	J_{\infty}(x) = \inf_{\pi \in \Pi^M} J_{\infty \pi}(x), \qquad x \in E.
\end{align}
The next lemma shows that the infinite horizon policy value \eqref{eq:infinite_policy_value} and value function \eqref{eq:opt_crit_infinite} are well-defined.

\begin{lemma}\label{thm:infinite_convergence}
	Under Assumption \ref{ass:infinite}, the sequence $\{J_{N\pi}\}_{N \in \N}$ converges pointwise for every Markov policy $\pi \in \Pi^M$ and the limit $J_{\infty \pi}$ is bounded by $\ubar \b$ and $\bar \b$.
\end{lemma}
\begin{proof}
	First, we show by induction that for all $N \in \N$
	\begin{align}\label{eq:infinite_convergence_1}
		J_{N\pi}(x) \geq J_{N-1 \pi}(x) + (\alpha\beta)^{N-1} \ubar b(x), \qquad x \in E.
	\end{align}
	For $N=1$ it holds by Assumption \ref{ass:infinite} (i) that $J_{1\pi}(x) \geq \ubar b(x) = J_{0\pi}(x) + (\alpha\beta)^{0} \ubar b(x)$. For $N \geq 2$ it follows
	\begin{align*}
		J_{N\pi}(x)&= \rho\Big( c\big(x,d_0(x),T(x,d_0(x),Z\big) + \beta J_{N-1 \vec \pi}\big(T(x,d_0(x),Z)\big) \Big)\\
		&\geq  \rho\Big( c\big(x,d_0(x),T(x,d_0(x),Z\big)  + \beta J_{N-2 \vec \pi}\big(T(x,d_0(x),Z)\big) + \beta (\alpha\beta)^{N-2} \ubar b \big( T(x,d_0(x),Z) \big) \Big)\\
		&\geq \rho\Big( c\big(x,d_0(x),T(x,d_0(x),Z\big)  + \beta J_{N-2 \vec \pi}\big(T(x,d_0(x),Z)\big) \Big)\\
		&\phantom{\geq} \ - \beta(\alpha\beta)^{N-2} \rho\Big( -\ubar b \big( T(x,d_0(x),Z) \big) \Big)\\
		&\geq \rho\Big( c\big(x,d_0(x),T(x,d_0(x),Z\big)  + \beta J_{N-2 \vec \pi}\big(T(x,d_0(x),Z)\big) \Big) + (\alpha\beta)^{N-1} \ubar b(x)\\
		&= J_{N-1\pi}(x) + (\alpha\beta)^{N-1} \ubar b(x).
	\end{align*} 
	The first inequality is by the induction hypothesis, the second one is by Lemma \ref{thm:subadditivity_complement} together with the positive homogeneity of $\rho$ and the third one is due to Assumption \ref{ass:infinite} (i). Thus, \eqref{eq:infinite_convergence_1} holds. Applying this inequality repeatedly for $N, N-1, \dots, m$ yields
	\begin{equation}\label{eq:weak_monotinicity}
		J_{N\pi}(x) \geq J_{m\pi}(x) + \sum_{k=m}^{N-1} (\alpha\beta)^k \ubar b(x) \geq  J_{m\pi}(x) +  \sum_{k=m}^{\infty} (\alpha\beta)^k \ubar b(x),
	\end{equation}
	where $\delta_m(x)=\sum_{k=m}^{\infty} (\alpha\beta)^k \ubar b(x)$ are non-positive functions with $\lim_{m \to \infty} \delta_m(x)=0$. Hence, the sequence of functions $\{J_{N\pi}\}_{N \in \N}$ is weakly increasing and therefore convergent to a limit function $J_{\infty \pi}$ by Lemma A.1.4 in \cite{BaeuerleRieder2011}. The global bounds \eqref{eq:infinite_global_bounding} also apply to the limit $J_{\infty \pi}$.
\end{proof}

\begin{lemma}\label{thm:contraction}
	Given Assumption \ref{ass:infinite}, the Bellman operator $\T$ is a contraction on $I=[\ubar \b, \bar \b ]$ with modulus $\alpha\beta \in (0,1)$.
\end{lemma}
\begin{proof}
	Let $v \in I$. It has been established in the proof of Theorem \ref{thm:finite} that $\T v$ is lower semicontinuous. Furthermore,
	\begin{align*}
		\T v(x) &\geq \T \ubar \b(x) =  \inf_{a \in D(x)} \rho\Big( c\big(x,a,T(x,a,Z)\big) + \frac{\beta}{1-\alpha\beta} \ubar b\big(T(x,a,Z)\big) \Big)\\
		&\geq \inf_{a \in D(x)} \rho\Big( c\big(x,a,T(x,a,Z)\big)\Big) - \frac{\beta}{1-\alpha\beta} \rho\Big(\!-\! \ubar b\big(T(x,a,Z)\big) \Big)\geq \ubar b(x) + \frac{\alpha\beta}{1-\alpha\beta}\ubar b(x) =\ubar \b(x).
	\end{align*}
	The second inequality is by Lemma \ref{thm:subadditivity_complement} together with the positive homogeneity of $\rho$ and the third one is due to Assumption \ref{ass:infinite} (i).  Regarding the upper bounding function one can argue similarly, using the subadditivity of $\rho$ instead of Lemma \ref{thm:subadditivity_complement}:
	\begin{align*}
		\T v(x) &\leq \T \bar \b(x) =\inf_{a \in D(x)} \rho\Big( c\big(x,a,T(x,a,Z)\big) + \frac{\beta}{1-\alpha\beta} \bar b\big(T(x,a,Z)\big) \Big)\\
		&\leq \inf_{a \in D(x)} \rho\Big( c\big(x,a,T(x,a,Z)\big)\Big) + \frac{\beta}{1-\alpha\beta} \rho\Big(\bar b\big(T(x,a,Z)\big) \Big) \leq \bar b(x) + \frac{\alpha\beta}{1-\alpha\beta}\bar b(x)=\bar \b(x).
	\end{align*}  
	Hence, the operator $\T$ is an endofunction on $I$ and it remains to verify the Lipschitz constant $\alpha\beta$. For $v_1,v_2 \in I$ it holds
	\begin{align*}
		\left\lvert \T v_1(x) - \T v_2(x) \right\rvert & \leq \sup_{a \in D(x)} \left\lvert Lv_1(x,a) - Lv_2(x,a)\right\rvert \\
		& \leq \beta \sup_{a \in D(x)} \rho\Big( \left\lvert  v_1\big(T(x,a,Z)\big) -  v_2\big(T(x,a,Z)\big) \right\rvert \Big)\\
		& \leq \beta \sup_{a \in D(x)} \rho\Big(  \|v_1-v_2\|_b   b\big(T(x,a,Z)\big) \Big)\\
		&=\beta \|v_1-v_2\|_b   \sup_{a \in D(x)} \rho\Big(   \bar b\big(T(x,a,Z)\big) - \ubar b\big(T(x,a,Z)\big)  \Big)\\
		&\leq\beta \|v_1-v_2\|_b   \sup_{a \in D(x)} \Big[\rho\Big(   \bar b\big(T(x,a,Z)\big)\Big) + \rho\Big( - \ubar b\big(T(x,a,Z)\big)  \Big) \Big]\\
		&\leq\alpha \beta \|v_1-v_2\|_b   \Big[\bar b(x) - \ubar b(x)\Big]\\
		& =\alpha\beta \|v_1-v_2\|_b  b(x).
	\end{align*}
	Dividing by $b(x)$ and taking the supremum over $x \in E$ on the left hand side completes the proof. Note that the second inequality is by Lemma \ref{thm:coherent_triangular}, the fourth one due to the subadditivity of $\rho$ and the last one by Assumption \ref{ass:infinite} (i).
\end{proof}

Under a finite planning horizon $N \in \N$ we have characterized the value function with the Bellman equation \eqref{eq:infinite_finite_Bellman}. We will show that this is compatible with the optimality criterion of the infinite horizon model \eqref{eq:opt_crit_infinite}. To this end, we define the \emph{limit value function}
\[ J(x)= \lim_{N \to \infty} J_N(x), \qquad x \in E. \]
Note that the limit exists since it follows from \eqref{eq:weak_monotinicity} that $J_N \ge J_m+\delta_m$ for all $N\ge m$ which implies the convergence.

\begin{theorem}\label{thm:infinite}
	Let Assumptions \ref{ass:continuity_compactness} and \ref{ass:infinite} be satisfied. Then it holds:
	\begin{enumerate}
		\item The limit value function $J$ is the unique fixed point of the Bellman operator $\T$ in $I=[\ubar \b, \bar \b]$.
		\item There exists a Markov decision rule $d^*$ such that $\T_{d^*} J = \T J$.
		\item Each stationary policy $\pi^*=(d^*,d^*,\dots)$ induced by a Markov decision rule $d^*$ as in b) is optimal for optimization problem \eqref{eq:opt_crit_infinite} and it holds $J_{\infty} = J$. 
	\end{enumerate}
\end{theorem}
\begin{proof}
	\begin{enumerate}
		\item The fact that $J$ is the unique fixed point of the operator $\T$ in $I$ follows directly from Banach's Fixed Point Theorem using Lemma \ref{thm:contraction}.
		\item The existence of a minimizing Markov decision rule follows from the respective result in the finite horizon case, cf.\ Theorem \ref{thm:finite}. 
		\item Let $d^*$ be a Markov decision rule as in part b) and $\pi^*=(d^*,d^*,\dots)$. Then it holds 
		$$J(x)  \leq J_{\infty}(x) \leq J_{\infty \pi^*}(x), \qquad x \in E.$$
		The second inequality holds by definition. Regarding the first one note that for any $\pi \in \Pi^M$ we have $J_N(x) \leq J_{N\pi}(x)$ for all $N \in \N_0$. Letting $N \to \infty$ yields $J(x) \leq J_{\infty \pi}(x)$. Since $\pi \in \Pi^M$ was arbitrary we get $ J(x) \leq \inf_{\pi \in \Pi^M} J_{\infty \pi}(x) = J_{\infty}(x)$. It remains to show
		\begin{align}\label{eq:infinite_proof_1}
			J_{\infty \pi^*}(x) \leq J(x), \qquad x \in E.
		\end{align} 
		To that end, we will prove by induction that for all $N \in \N_0$ and $x \in E$
		\begin{align}\label{eq:infinite_proof_2}
			J(x) \geq J_{N\pi^*}(x) + \frac{(\alpha\beta)^N}{1-\alpha\beta} \ubar b(x).
		\end{align}
		Letting $N \to \infty$ in \eqref{eq:infinite_proof_2} yields \eqref{eq:infinite_proof_1} and concludes the proof. For $N=0$ equation \eqref{eq:infinite_proof_2} reduces to $J(x) \geq \frac{1}{1-\alpha\beta} \ubar b(x)= \ubar \b(x)$, which holds by part a). For $N \geq 1$ the induction hypothesis yields
		\begin{align*}
			J(x) &= \T_{d^*} J(x) \geq \T_{d^*} \left(J_{N-1\pi^*} + \frac{(\alpha\beta)^{N-1}}{1-\alpha\beta} \ubar b\right)(x) \\
			&\geq \rho\Big( c\big(x,d^*(x),T(x,d^*(x),Z)\big) + \beta  J_{N-1\pi^*}\big(T(x,d^*(x),Z)\big)\Big)\\
			&\phantom{\geq}\ - \beta\frac{(\alpha\beta)^{N-1}}{1-\alpha\beta} \rho\Big(- \ubar b \big(T(x,d^*(x),Z)\big) \Big)\\
			&\geq  J_{N\pi^*}(x) + \frac{(\alpha\beta)^{N}}{1-\alpha\beta} \ubar b(x).
		\end{align*}
		The second inequality is by Lemma \ref{thm:subadditivity_complement} together with the positive homogeneity of $\rho$ and the last one is by Assumption \ref{ass:infinite} (i). \qedhere
	\end{enumerate}
\end{proof}

Let us now consider the special case that the one-stage cost is bounded, i.e.
\begin{enumerate}
	\item[(B)] there exist $\ubar b \in \R_-$ and $\bar b \in \R_+$ such that $b=\bar b - \ubar b >0$ and $\ubar b \leq c\big(x,a,T(x,a,Z)\big) \leq \bar b$ $\P$-f.s.\ for all $(x,a) \in D$.
\end{enumerate}
Then, Assumption \ref{ass:infinite} (i) is satisfied with $\alpha =1$ and part (ii) is obvious. Part (iv) of the assumption reduces to $\beta<1$.

\begin{corollary}\label{thm:infinite_bounded_case}
	Given (B), Lemmata \ref{thm:infinite_convergence}, \ref{thm:contraction} and in case Assumption \ref{ass:continuity_compactness} is satisfied,  Theorem \ref{thm:infinite} hold for any normalized monetary risk measure with the Fatou property.
\end{corollary}
\begin{proof}
	The steps in the proofs that where justified by Lemma \ref{thm:subadditivity_complement}, subadditivity, positive homogeneity or Assumption \ref{ass:infinite} (i) now hold due to translation invariance and normalization. Nothing else has to be changed.
\end{proof}

\section{Connection to Distributionally Robust MDP}\label{sec:robust}

We consider the stationary version of the decision model with no terminal cost under both finite and infinite horizon in this section. If the planning horizon is finite, stationarity is only assumed for convenience and everything can be transferred to a non-stationary setting purely by notational changes. Let the risk measure $\rho$ be proper and coherent with the Fatou property. By inserting the dual representation of Proposition \ref{thm:coherent_risk_measure_dual} in the Bellman equation, we get
\begin{align*}
J_N(x) &= 0,\\
J_n(x) &= \inf_{a \in D(x)} \sup_{\Q \in \Qc} \E^\Q\Big[ c\big(x,a,T(x,a,Z)\big) + \beta  J_{n+1}\big(T(x,a,Z)\big) \Big], \qquad x \in E,
\end{align*}
i.e.\ the Bellman equation of a distributionally robust MDP as considered in \cite{BauerleGlauner2020}. Under some minor technical assumptions we have indeed a special case of the distributionally robust MDP and thus obtain a global interpretation of the recursively defined risk-sensitive optimality criterion. 

Due to the independence of the disturbances we can w.l.o.g.\ assume that the underlying probability space has a product structure $(\Omega,\A,\P) = \bigotimes_{n=1}^\infty (\Omega_1,\A_1,\P_1)$ with $Z_n(\omega)=Z_n(\omega_n)$ only depending on component $\omega_n$ of $\omega=(\omega_1,\omega_2,\dots) \in \Omega$. For a policy $\pi=(d_0,d_1,\dots) \in \Pi^M$ of the controller and $\gamma=(\gamma_0,\gamma_1,\dots)$, where $\gamma_n:D \to \Qc$ is measurable, we define the transition kernel
\begin{align*}
	Q_n^{\pi\gamma}(B|x,a)= \int \1_B\big(T(x,d_n(x),Z(\omega))\big) \gamma_n(\dif \omega|x,d_n(x)), \qquad B \in \B(E), \ x \in E,
\end{align*}
and the law of motion $\Q^{\pi\gamma}_x = \delta_x \otimes Q_0^{\pi\gamma} \otimes Q_1^{\pi\gamma} \otimes \dots$ The set of all possible laws of motion under policy $\pi \in \Pi^M$ is denoted by $\Qf_\pi = \{ \Q_x^{\pi \gamma}: \gamma \in \Gamma \}$ with $\Gamma$ being the set of all possible $\gamma$.

\begin{theorem} \phantomsection \label{thm:abstract_connection_robust_risk}
	Let Assumption \ref{ass:infinite} be fulfilled with the following tightening in part (i):
		\begin{align*}
		\rho\big( c^-(x,a,T(x,a,Z))\big)& \leq -\ubar b(x), &  \rho\big( c^+(x,a,T(x,a,Z))\big) &\leq \bar b(x), & (x,a) \in D.
		\end{align*}
		Furthermore, let the underlying probability space have a product structure as above and let the probability measure $\P_1$ on $(\Omega_1,\A_1)$ be separable. Then, for $N \in \N \cup \{\infty\}$ it holds
		\begin{align}\label{eq:robust_1}
			J_N(x) = \inf_{\pi \in \Pi^M} \sup_{\Q \in \Qf_\pi} \E^{\Q}\left[ \sum_{k=0}^{N-1} \beta^{k} c(X_k,d_k(X_k),X_{k+1}) \right]
		\end{align}
\end{theorem}
\begin{proof}
		We need to verify Assumption 3.1 in \cite{BauerleGlauner2020}. Then, the assertion follows from Theorem 3.10 therein for the finite horizon case and from Theorem 4.18 in \cite{Glauner2020} for the infinite horizon case. Part (i) holds since we have for all $\Q \in \Qc$ and $(x,a) \in D$
		\begin{align*}
		\E^\Q\left[-c^-(x,a,T(x,a,Z))\right] &\geq \inf_{\Q \in \Qc} \E^\Q\left[-c^-(x,a,T(x,a,Z))\right]= - \sup_{\Q \in \Qc} \E^\Q\left[c^-(x,a,T(x,a,Z))\right]\\
		&= -\rho\big( c^-(x,a,T(x,a,Z)) \big)\geq \ubar b(x),\\
		\E^\Q\left[\ubar b(T(x,a,Z)) \right] &\geq \inf_{\Q \in \Qc} \E^\Q\left[\ubar b(T(x,a,Z))\right]= -\sup_{\Q \in \Qc} \E^\Q\left[-\ubar b(T(x,a,Z))\right]\\
		&= -\rho\big( -\ubar b(T(x,a,Z)) \big)\geq \alpha \ubar b(x),\\
		\E^\Q\left[c^+(x,a,T(x,a,Z))\right] &\leq \sup_{\Q \in \Qc} \E^\Q\left[c^+(x,a,T(x,a,Z))\right]= \rho\big( c^+(x,a,T(x,a,Z)) \big) \leq \ubar b(x),\\
		\E^\Q\left[\bar b(T(x,a,Z)) \right] &\leq \sup_{\Q \in \Qc} \E^\Q\left[\bar b(T(x,a,Z))\right]= \rho\big( \bar b(T(x,a,Z)) \big)\leq \alpha \bar b(x).
		\end{align*}
		Part (ii) equals Assumption \ref{ass:infinite} (ii). Finally, part (iii) holds since $\rho(X)= \max_{\Q \in \Qc} \E^\Q[X]$ by Proposition \ref{thm:coherent_risk_measure_dual} where $\Qc \subseteq \M_1^q(\Omega,\A,\P)$ is weak* compact and therefore norm bounded by the Banach-Alaoglu Theorem 6.21 in \cite{AliprantisBorder2006}. 
\end{proof}

It is readily checked that for a fixed policy $\pi \in \Pi^M$ of the controller $\tilde \rho(X)= \sup_{\Q \in \Qf_\pi} \E^\Q[X]$, $X \in L^p(\Omega,\A,\P)$, defines a coherent risk measure. If the stage-wise applied risk measure $\rho$ is spectral and the model data has certain monotonicity properties, one can choose $\Qf$ to be independent of $\pi$, cf.\ Lemma 6.8 and subsequent remarks in \cite{BauerleGlauner2020}. In this case, the recursive minimization of spectral risk measures is equivalent to the minimization of a non-standard coherent risk measure applied to the total cost.  

Besides, one can reformulate \eqref{eq:robust_1} to
\begin{align*}
J_N(x) = \inf_{\pi \in \Pi^M} \sup_{\gamma \in \Gamma} \E^{\pi \gamma}\left[ \sum_{k=0}^{N-1} \beta^{k} c(X_k,d_k(X_k),X_{k+1})\right]
\end{align*}
with the interpretation of a Stackelberg game of the controller against a theoretical opponent (nature) selecting the most adverse disturbance distribution in each scenario. Here, $\gamma=(\gamma_0,\gamma_1,\dots)$ with $\gamma_n:D \to \Qc$ is a Markov policy of nature. This game is extensively studied in \cite{BauerleGlauner2020}. From this perspective we get another global interpretation of the recursively defined objective function as robust minimization of the expected total cost.

\section{Relaxed Assumptions for Monotone Models}\label{sec:monotone}

The  model has been introduced in Section \ref{sec:decision_model} with a general Borel space as state space.  However, in many applications the state space is simply $\R$. In this case, the assumption on the transition function can be relaxed to semicontinuity  when the transition and one-stage cost function have some form of monotonicity. For notational convenience, we consider the stationary model with no terminal cost under both finite and infinite horizon in this section. We replace Assumption \ref{ass:continuity_compactness} by 

\begin{assumption}\label{ass:monotone}
	\begin{enumerate}
		\item[(i)] The state space is the real line $E=\R$.
		\item[(ii)]  The sets $D(x)$ are compact and  $\R \ni x \mapsto D(x)$ is upper semicontinuous and decreasing, i.e.\ $D(x) \supseteq D(y)$ for $x \leq y$.
		\item[(iii)] The transition functions $T$ is lower semicontinuous in $(x,a)$ and increasing in $x$.
		\item[(iv)] The one-stage cost function $c$ is lower semicontinuous in $(x,a,x')$ and increasing in $(x,x')$. 
	\end{enumerate}
\end{assumption}

How do the modified continuity assumptions affect the validity of the results in Sections \ref{sec:finite} and \ref{sec:infinite}? Lemmata \ref{thm:stagewise_bound}, \ref{thm:Lpbound}, \ref{thm:infinite_convergence} and \ref{thm:contraction} were proven without using the continuity of $T$. Thus, only Theorems \ref{thm:finite} and \ref{thm:infinite} need to be looked at.

\begin{proposition}\label{thm:monotone}
	Let the new continuity and monotonicity Assumptions \ref{ass:monotone} be satisfied. Then,
	\begin{enumerate}
		\item under Assumption \ref{ass:finite}, the assertion of Theorem \ref{thm:finite} remains true.
		\item under Assumption \ref{ass:infinite}, the assertion of Theorem \ref{thm:infinite} remains true.
	\end{enumerate}
	 \noindent In both cases, the value functions are increasing and the set of potential value functions can be replaced by $\BB = \{ v \in \BB_\b : \ v \text{ lower semicontinuous and increasing}  \}$.
\end{proposition}
\begin{proof}
	In the proof of Theorem \ref{thm:finite}, the continuity of $T$ is only used to show that $D \ni (x,a) \mapsto Lv(x,a)$ is lower semicontinuous for every $v \in \BB$. Due to the monotonicity assumptions, $$D \ni (x,a) \mapsto c\big(x,a,T(x,a,Z(\omega))\big)+ \beta v\big(T(x,a,Z(\omega))\big)$$ is lower semicontinuous for every $\omega \in \Omega$. Now, the lower semicontinuity of $D \ni (x,a) \mapsto Lv(x,a)$ and the existence of a minimizing decision rule follow as in the proof of Theorem \ref{thm:finite}. The fact that $\T v$ is increasing for every $v \in \BB$ follows as in Theorem 2.4.14 in \cite{BaeuerleRieder2011}. Theorem \ref{thm:infinite} uses the continuity of $T$ only indirectly through Theorem \ref{thm:finite}.
\end{proof}

With the real line as state space, a simple separation condition is sufficient for Assumptions \ref{ass:finite} (ii) or \ref{ass:infinite} (ii).

\begin{lemma}\label{thm:separation}
	Let there be upper semicontinuous functions $\vartheta_{1},\vartheta_{2}:D \to \R_+$ and measurable functions $\Theta_{1},\Theta_{2}: \Z \to \R_+$ which fulfill $\Theta_{1}(Z),\Theta_{2}(Z) \in L^p(\Omega,\A,\P)$ and
	\begin{align*}
	|c(x,a,T(x,a,z))| \leq \vartheta_{1}(x,a) + \Theta_{1}(z),  \qquad  b(T(x,a,z)) \leq \vartheta_{2}(x,a) + \Theta_{2}(z)
	\end{align*}
	for every $(x,a,z) \in D \times \Z$. Then Assumptions \ref{ass:finite} (ii) and \ref{ass:infinite} (ii) are satisfied.
\end{lemma}
\begin{proof}
	Let $(\bar x, \bar a) \in D$. We can choose $\epsilon >0$ arbitrarily. The set $S=[\bar x- \epsilon, \bar x+\epsilon] \times D(\bar x - \epsilon)$ is compact w.r.t.\ the product topology by the Tychonoff Product Theorem 2.61 in \cite{AliprantisBorder2006}. Moreover, $B_\epsilon(\bar x,\bar a) \cap D \subseteq S$ since the set-valued mapping $D(\cdot)$ is decreasing. Due to upper semicontinuity there exist $(x_i, a_i) \in S$ such that $\vartheta_i( x_i, a_i) = \sup_{(x,a) \in S} \vartheta_i(x,a)$, $i=1,2$. Hence, one can define
	\[ \Theta_i^{\bar x, \bar a}(\cdot) = \vartheta_i(x_i, a_i) + \Theta_i(\cdot), \qquad i=1,2 \]
	and Assumptions \ref{ass:finite} (ii) and \ref{ass:infinite} (ii) are satisfied.
\end{proof}

A monotone model not only allows for weaker assumptions on the transition function, but also requirements regarding the risk measure may be relaxed. In the following, we study two such cases: local bounding and infinite horizon cost minimization with bounded below cost. 

Firstly, the existence of a global upper and lower bounding function can be guaranteed by suitable local bounding functions as in Lemma \ref{thm:stagewise_bound}. However due to the monotonicity properties of the model, the risk measure does not need to be coherent. E.g.\ spectral risk measures can be replaced by other distortion risk measures.

\begin{lemma}\label{thm:monotone_stagewise_bound}
	Let Assumption  \ref{ass:monotone} be satisfied and the monetary risk measure $\rho$ be positive homogeneous and comonotonic additive. If there exist  $\ubar \epsilon, \bar \epsilon \geq 0$ with $\ubar \epsilon+ \bar \epsilon =1$, increasing functions $\ubar b:\R \to (-\infty, -\ubar \epsilon]$, $\bar b: \R \to [\bar \epsilon,\infty)$ and a constant $\alpha > 0 $ such that $\alpha\beta \in (0,1)$ and
	\begin{align*}
	\rho\big( c(x,a,T(x,a,Z))\big)& \geq \ubar b(x), & \rho\big( \ubar b(T(x,a,Z))\big) &\geq \alpha \ubar b(x),\\
	\rho\big( c(x,a,T(x,a,Z))\big) &\leq \bar b(x), & \rho\big( \bar b(T(x,a,Z))\big) &\leq \alpha \bar b(x),
	\end{align*}
	for all $(x,a) \in D$, then
	\begin{align*}
	\ubar \b = \frac{1}{1-\alpha\beta}\ubar b \qquad \text{and} \qquad \bar \b = \frac{1}{1-\alpha\beta}\bar b
	\end{align*}
	are global lower/ upper bounding functions and Assumption \ref{ass:finite} (i) holds.
\end{lemma}
\begin{proof}
	We proceed by backward induction. At time $N$ there is nothing to show. Assuming the assertion holds at time $n+1$, it follows for time $n$:
	\begin{align*}
	V_{n\pi}(h_n) &= \rho\Big( c\big(x_n,d_n(h_n),T(x_n,d_n(h_n),Z)\big) + \beta V_{n+1\pi}\big(h_n,d_n(h_n),T(x_n,d_n(h_n),Z)\big) \Big)\\
	&\geq \rho\Big( c\big(x_n,d_n(h_n),T(x_n,d_n(h_n),Z)\big) + \frac{\beta}{1-\alpha\beta} \ubar b\big(T(x_n,d_n(h_n),Z)\big) \Big)\\
	&= \rho\Big( c\big(x_n,d_n(h_n),T(x_n,d_n(h_n),Z)\big)\Big) + \frac{\beta}{1-\alpha\beta} \rho \Big( \ubar b\big(T(x_n,d_n(h_n),Z)\big) \Big)\\
	&\geq \ubar b(x_n) + \frac{\alpha\beta}{1-\alpha\beta} \ubar b(x_n)= \ubar \b(x_n),
	\end{align*}
	$\pi \in \Pi, \ h_n \in \H_n$. The second equality is by the comonotonic additivity and positive homogeneity of $\rho$. Regarding the upper bounding function one argues analogously. 	
\end{proof}

In Lemma \ref{thm:monotone_stagewise_bound}, the local bounding functions are assumed to be increasing, which was not necessary in Lemma \ref{thm:stagewise_bound}. Also note that we only have to require $\rho\big( \ubar b(T(x,a,Z))\big) \geq \alpha \ubar b(x)$, $(x,a) \in D$ which is weaker than the corresponding assumption for the model with general state space, cf. Lemma \ref{thm:stagewise_bound} and Remark \ref{rem:stagewise_bound}.

As a second example, where the assumptions on the risk measure can be relaxed, we consider infinite horizon cost minimization with bounded below cost. For absolutely bounded cost functions we already showed in Corollary \ref{thm:infinite_bounded_case} that a coherent risk measure is not necessary to solve the infinite horizon problem. This result is very general regarding the risk measure but very restrictive concerning the one-stage cost. The monotone model allows for a middle course.

\begin{enumerate}
	\item[(B$^-$)] There exist $\ubar b \leq 0$, $\bar \epsilon \geq 0$ and $\alpha \geq 1$ with $\bar \epsilon - \ubar b=1$ and an increasing function $\bar b:\R \to [\bar \epsilon,\infty)$ such that $c\big(x,a,T(x,a,Z)\big) \geq \ubar b $ $\P$-f.s.\ and
	\begin{align*}
	\rho\big( c(x,a,T(x,a,Z))\big) &\leq \bar b(x), & \rho\big( \bar b(T(x,a,Z))\big) &\leq \alpha \bar b(x).
	\end{align*}
	for all $(x,a) \in D$.
\end{enumerate}

W.l.o.g.\ we assume $\alpha \geq 1$ since then $\rho(-\ubar b) = -\ubar b \leq \alpha \ubar b$ due to translation invariance and normalization. Otherwise one would need separate alphas for the lower and upper local bounding function. If the risk measure is comonotonic additive and positive homogeneous, the objective function is globally bounded under (B$^-$) due to Lemma \ref{thm:monotone_stagewise_bound} and Theorem \ref{thm:finite} remains true. Under an infinite planning horizon, the assertion of Theorem \ref{thm:infinite} can be proven without requiring a coherent risk measure. When we refer to the interval $I=[\ubar \b, \bar \b]$ in the following, it is to be understood as a subset of the modified function space $\BB$ as in Proposition \ref{thm:monotone}.

\begin{proposition}\label{thm:monotone_bounded_below}
	Let Assumptions \ref{ass:monotone} and \ref{ass:infinite} be satisfied with the modification that part (i) is replaced by (B$^-$) and part (iii) by the requirement that $\rho$ is a law invariant, comonotonic additive and positive homogeneous monetary risk measure with the Fatou property. Then it holds:
	\begin{enumerate}
		\item The sequence $\{J_{N\pi}\}_{N \in \N}$ converges pointwise for every Markov policy $\pi \in \Pi^M$ and the limit function $J_{\infty \pi}$ is bounded by $\ubar \b$ and $\bar \b$.
		\item The Bellman operator $\T$ is a contraction on $I$ with modulus $\alpha\beta \in (0,1)$ and the limit value function $J$ is the unique fixed point of $\T$ in $I$.
		\item There exists a Markov decision rule $d^*$ such that $\T_{d^*} J = \T J$, each stationary policy $\pi^*=(d^*,d^*,\dots)$ induced by such a Markov decision rule is optimal for optimization problem \eqref{eq:opt_crit_infinite} and it holds $J_{\infty} = J$.
	\end{enumerate}
\end{proposition}
\begin{proof}
	\begin{enumerate}
		\item We show by induction that for all $N \in \N$
		\begin{align}\label{eq:real_infinite_convergence_1}
		J_{N\pi}(x) \geq J_{N-1 \pi}(x) + (\alpha\beta)^{N-1} \ubar b, \qquad x \in \R.
		\end{align}
		For $N=1$ it holds due to (B$^-$) that $J_{1\pi}(x) \geq \ubar b = J_{0\pi}(x) + (\alpha\beta)^{0} \ubar b.$
		For $N \geq 2$ it follows with the monotonicity and translation invariance of $\rho$ that
		\begin{align*}
		J_{N\pi}(x)&= \T_{d_0} J_{N-1 \vec \pi}(x) \geq \T_{d_0} \left(J_{N-2 \vec \pi}+ (\alpha\beta)^{N-2} \ubar b \right)(x) = \T_{d_0} J_{N-2 \vec \pi}(x) + \beta (\alpha\beta)^{N-2} \ubar b\\
		& \geq \T_{d_0} J_{N-2 \vec \pi}(x) + (\alpha\beta)^{N-1} \ubar b = J_{N-1\pi}(x) + (\alpha\beta)^{N-1} \ubar b.
		\end{align*} 
		Thus, \eqref{eq:real_infinite_convergence_1} holds. Applying this inequality repeatedly for $N, N-1, \dots, m$ yields
		\begin{align*}
		J_{N\pi}(x) \geq J_{m\pi}(x) + \sum_{k=m}^{N-1} (\alpha\beta)^k \ubar b \geq J_{m\pi}(x) + \sum_{k=m}^{\infty} (\alpha\beta)^k \ubar b.
		\end{align*}
		Since $\sum_{k=m}^{\infty} (\alpha\beta)^k \ubar b$ is non-positive and converges to zero as $m \to \infty$, the sequence $\{J_{N\pi}\}_{N \in \N}$ is weakly increasing and hence convergent to a limit $J_{\infty \pi}$ by Lemma A.1.4 in \cite{BaeuerleRieder2011}. Clearly, the global bounds $\ubar \b, \bar \b(\cdot)$ also apply to the limit $J_{\infty \pi}$.
		\item Let $v \in I$. Due to Proposition \ref{thm:monotone} $\T v$ is increasing and lower semicontinuous. Furthermore, the monotonicity and translation invariance of $\rho$ imply
		\begin{align*}
		\T v(x) \geq \T \ubar \b(x) = \T 0(x) + \ubar \b \geq \ubar b + \frac{\alpha\beta}{1-\alpha\beta}\ubar b =\ubar \b.
		\end{align*}
		Regarding the upper bounding function it follows from the comonotonic additivity and positive homogeneity of $\rho$ that
		\begin{align*}
		\T v(x) &\leq \T \bar \b(x) = \inf_{a \in D(x)} \rho\Big( c\big(x,a,T(x,a,Z)\big) + \frac{\beta}{1-\alpha\beta} \bar b\big(T(x,a,Z)\big) \Big)\\
		&= \inf_{a \in D(x)} \rho\Big( c\big(x,a,T(x,a,Z)\big)\Big) + \frac{\beta}{1-\alpha\beta} \rho\Big(\bar b\big(T(x,a,Z)\big) \Big)\\
		&\leq \bar b(x) + \frac{\alpha\beta}{1-\alpha\beta}\bar b(x)=\bar \b(x).
		\end{align*}
		I.e.\ $\T$ is an endofunction on $I$ and it remains to verify the Lipschitz constant. For $v_1,v_2 \in I$ it holds
		\begin{align*}
		&\T v_1(x) - \T v_2(x)\\
		&\leq \sup_{a \in D(x)}  Lv_1(x,a) - Lv_2(x,a)= \beta \sup_{a \in D(x)}  \rho\Big( v_1\big(T(x,a,Z)\big) \Big) - \rho\Big( v_2\big(T(x,a,Z)\big) \Big) \\
		& = \beta \sup_{a \in D(x)}  \rho\Big(  v_1\big(T(x,a,Z)\big) - v_2\big(T(x,a,Z)\big) + v_2\big(T(x,a,Z)\big) \Big) - \rho\Big( v_2\big(T(x,a,Z)\big) \Big)\\
		& \leq \beta \sup_{a \in D(x)}  \rho\Big( \|v_1-v_2\|_b b\big(T(x,a,Z)\big)  + v_2\big(T(x,a,Z)\big) \Big) - \rho\Big( v_2\big(T(x,a,Z)\big) \Big) \\
		&=  \|v_1-v_2\|_b \beta \sup_{a \in D(x)}  \rho\Big(  b\big(T(x,a,Z)\big)\Big)= \|v_1-v_2\|_b \beta \sup_{a \in D(x)} \Big[ \rho\Big(\bar  b\big(T(x,a,Z)\big)\Big) - \ubar b \Big]\\
		& \leq\alpha\beta \|v_1-v_2\|_b  [\bar b(x)-\ubar b]=\alpha\beta \|v_1-v_2\|_b  b(x).
		\end{align*}
		The first equality is by comonotonic additivity and positive homogeneity. Since $\ubar b$ is constant, $b(\cdot)=\bar b(\cdot) - \ubar b$ is an increasing function and so is $v_2$. Therefore, the third equality is again by comonotonic additivity. The last inequality is by (B$^-$) using $\alpha \geq 1$. Interchanging the roles of $v_1$ and $v_2$ yields
		\[ 	\left|\T v_1(x) - \T v_2(x)\right| \leq \alpha\beta \|v_1-v_2\|_b  b(x). \]
		Finally, dividing by $b(x)$ and taking the supremum over $x \in \R$ shows that $\T$ is a contraction and Banach's Fixed Point Theorem yields the assertion. 
		\item The existence of a minimizing Markov decision rule follows from Proposition \ref{thm:monotone}. With the same argument as in the proof of Theorem \ref{thm:infinite}, the relation $J  \leq J_{\infty} \leq J_{\infty \pi}$ holds for any policy and it remains to show that $J_{\infty \pi^*} \leq J$ for the specific policy $\pi^*$. To that end, we will prove by induction that $ J \geq J_{N\pi^*} + (\alpha\beta)^N \ubar \b$ for all $N \in \N_0$. Then, letting $N \to \infty$ concludes the proof. The case $N=0$, i.e.\ $J(x) \geq \frac{1}{1-\alpha\beta} \ubar b$, holds by part b). For $N \geq 1$ we have
		\begin{align*}
		J(x) &= \T_{d^*} J(x) \geq \T_{d^*} \left(J_{N-1\pi^*} + (\alpha\beta)^{N-1} \ubar \b \right)(x) = \T_{d^*} J_{N-1\pi^*}(x) + \beta(\alpha\beta)^{N-1} \ubar \b\\
		&\geq \T_{d^*} J_{N-1\pi^*}(x) + (\alpha\beta)^{N} \ubar \b = J_{N\pi^*}(x) + (\alpha\beta)^N \ubar \b.
		\end{align*}
		The first inequality is by the induction hypothesis and the monotonicity of $\rho$, the equality thereafter is by translation invariance and the second inequality holds since $\alpha \geq 1$. \qedhere
	\end{enumerate}
\end{proof}

\section{Examples}\label{sec:examples}
In this section, we present some applications of the results in the previous sections. In particular we show that often structural results about optimal policies which are know from the classical iterated expectation case still hold under more general risk measures.

\begin{example}[Value-at-Risk is Myopic in Monotone Models]
	In a monotone model as in Section \ref{sec:monotone}, where the one-stage cost function does not depend on the controller's action, i.e.\ $c(x,a,x')=c(x,x')$, recursive decision making with Value-at-Risk is myopic. This can be seen as follows. Let Assumptions \ref{ass:monotone} and \ref{ass:finite} (i),(ii) be satisfied.  The Bellman equation here reads
	\begin{align*}
		J_N(x)&=0\\
		J_n(x)&= \inf_{a \in D(x)} \VaR_{\alpha}\big( c(x,T(x,a,Z)) + \beta J_{n+1}(T(x,a,Z)) \big), \qquad n=0,\dots,N-1.
	\end{align*}
	We can now interchange $\VaR$ with the increasing lower semicontinuous (i.e.\ left-continuous) function $h(x') = c(x,x') + \beta J_{n+1}(x'), \ x' \in \R$ by properties of the quantile function (see e.g. Proposition 2.2 in \cite{BauerleGlauner2018}). Doing this we obtain
	\[ J_n(x) = \inf_{a \in D(x)} h\big(\VaR_{\alpha}(T(x,a,Z))\big) =  h\left( \inf_{a \in D(x)} \VaR_{\alpha}(T(x,a,Z))\right). \]
	Hence, the minimizer of  $a \mapsto  \VaR_{\alpha}(T(x,a,Z))$ induces an optimal decision rule for each stage. In particular the optimal policy is stationary and does not depend on time.
	
	Note here that we can interpret a spectral risk measure as a Value-at-Risk criterion with unknown parameter $\alpha$ which has a prior distribution given by the density $\phi$. However, since we apply it recursively at each stage, learning of the parameter is not possible.
\end{example}

\begin{example}[Stopping Problems]
Let us consider the following standard stopping problem: Suppose a real-valued Markov chain $(X_n)\subset L^p$ is given by $X_{n+1}=T(X_n,Z_{n+1})$ where $(Z_n)$ is an i.i.d.\ sequence of random variables. We are allowed to observe the Markov chain and when we stop it in state $x$ we have to pay the cost $c(x)$. In case we do not stop we have to pay the fixed cost $\bar{c}$. We have to stop no later than time point $N$.  Suppose Assumptions \ref{ass:finite} (i) and (ii) are fulfilled.  The risk measure $\rho$ is simply monetary and finite. The Fatou property is not needed here since the existence of minimizers is immediate. The Bellman equation is 
	\begin{align*}
		J_N(x)&=x,\\
		J_n(x)&= \min \left\{ \rho(c(x)) ;\; \rho\left( \bar{c} +  \beta J_{n+1}\big(T(x,Z_{n+1})\big)\right) \right\}, \\
		&=  \min \left\{ \rho(c(x)) ;\; \bar{c} +  \rho\left( \beta J_{n+1}\big(T(x,Z_{n+1})\big)\right) \right\}, \qquad n=0,\dots,N-1.
	\end{align*}
In the well-known house selling application for example $X_n$ is the offer for a house at time $n$ that we may buy. When we decide to buy it we have to pay the price, i.e.\ $c(x)=x$. In case we do not buy, we still have to pay the rent $\bar{c}$. Offers are here assume to be i.i.d. Thus, the Bellman equation specializes to 
	\begin{align*}
		J_N(x)&=x,\\
		J_n(x)&= \min \left\{ \rho(x) ; \bar{c} + \rho\left( \beta J_{n+1}(Z_{n+1})\right) \right\}, \qquad n=0,\dots,N-1.
	\end{align*}
 Thus, when we define
$$ t_n := \sup \big\{ x\in \R :  \rho(x) \le \bar{c} + \rho\left( \beta J_{n+1}(Z_{n+1})\right)\Big\} $$ then the optimal policy obviously is to buy at time $n$ if $x\le t_n$, otherwise not. Hence the optimal strategy is still a threshold policy, but the thresholds depend on $\rho$.
For example if $\rho$ is normalized and $\rho(X) \ge \E X$ (this is e.g.\ satisfied for Average-Value-at-Risk or the Entropic risk measures) then $t_n \ge t^E_n$ where $t^E_n$ belongs to the case $\rho=\E$. Hence, under the risk measure we will accept an offer earlier.
\end{example}

\begin{example}[Casino Game]
Suppose we have to play $N$-times the same game and decide how much of our current capital we should bet. Outcomes are either a gain or a loss and given by i.i.d.\ random variables $(Z_n)$, with $\Pop(Z_n=1)=p=1-\Pop(Z_n=-1)$. Note that the $Z_n$ are bounded. We assume that the risk measure is monetary, law-invariant, positive homogeneous and has the Fatou property. We want to minimize the risk of a loss. Note that Assumptions \ref{ass:continuity_compactness} and \ref{ass:finite} are satisfied. The Bellman equation is 
	\begin{align*}
		J_N(x)&=-x,\\
		J_n(x)&= \inf_{0\le a\le x} \rho \left(  J_{n+1}\big(T(x,a,Z_{n+1})\big) \right), \\
		&=   \inf_{0\le a\le x} \rho \left( J_{n+1}\big(x+aZ\big) \right), \qquad n=0,\dots,N-1.
	\end{align*}
	It is then easy to see by induction that the optimal policy is stationary and given by
	$$ d^*(x) = \left\{ \begin{array}{cc}
	            0, & \rho(-Z) \ge 0,\\
	            x, & \rho(-Z) < 0. \end{array} \right.
	$$
From the monotonicity and law-invariance of $\rho$ it follows that there exists $p^* \in [0,1]$ such that 
	$$ d^*(x) = \left\{ \begin{array}{cc}
	0, & p < p*,\\
	x, & p \ge p^*. \end{array} \right.
	$$
In case $\rho(-Z) < 0 $, bold-play is best and we obtain by induction that $J_n(x) = -x(1-\rho(-Z))^{N-n}$.
When $\rho$ is additionally convex (which then means coherent) and the game is fair ($p=\frac12$), we obtain since $0\le_{cx} -Z_n$ in convex order implying that $0=\rho(0)\le \rho(-Z)$ (see Theorem 3.4 in \cite{BaeuerleMueller2006}) and it is obviously best not to play.

\end{example}

\begin{example}[Cash Balance]
In a cash balance problem the aim is to keep the cash level of a company close to zero, because a negative cash level means we have to pay interest and a positive cash level creates opportunity cost (see Section 2.6.2 in \cite{BaeuerleRieder2011}). We assume that a convex function $L:\R \to\R_+$ with $L(0)=0$ gives the cost of deviating from zero. The cash level is subject to random changes which are modelled as i.i.d.\ random variables $(Z_n)$. It is possible to increase or decrease the cash level at the beginning of each period by paying transfer cost. The transfer cost $c:\R\to\R_+$ are assumed to be piecewise linear:
$$ c(x') = c_u (x')^+ + c_d (x')^-$$
with $c_u, c_d >0$.  Of course the state is here the cash level and we choose the action to be the new cash level. Hence the transition function is $T(x,a,z)= a-z$. We consider the infinite horizon problem and suppose that Assumption \ref{ass:infinite} is in force. Assumption \ref{ass:continuity_compactness} is here satisfied, except for the compactness of the admissible actions which are here $\R$. However, it can be seen that it is possible to restrict to a compact level set (see Section 2.6.2 in \cite{BaeuerleRieder2011} for details). The Bellman equation for the infinite horizon problem is here
\begin{eqnarray*}
 J_\infty(x) &=& \sup_{a\in \R} \rho\Big( c(a-x)+L(a)+\beta J_\infty(a-Z)\Big)\\
 &=&  \sup_{a\in \R} \left\{ c(a-x)+L(a)+ \beta \rho\big( J_\infty(a-Z)\big)\right\}
\end{eqnarray*}
Now we can proceed exactly in the same way as in the classical case  (see Section 2.6.2 in \cite{BaeuerleRieder2011} for details) since the functions
\begin{eqnarray*}
h_u(a) &:=& (a-x)c_u + L(a)+ \beta \rho\big( J_\infty(a-Z)\big)\\
h_d(a) &:=& (x-a)c_d+\beta \rho\big( J_\infty(a-Z)\big)
\end{eqnarray*}
are still both convex under our assumption that $\rho$ is convex. We obtain:

\begin{proposition}
For the cash balance problem with infinite horizon it holds under Assumption \ref{ass:infinite}:
\begin{enumerate}
  \item[a)] There exist critical levels $S_{-}$ and $ S_{+}$ such that
$$ J_\infty(x) = \left\{ \begin{array}{cl}
(S_--x) c_u + L(S_{-})+\beta \rho\big( J_\infty(S_{-}-Z)\big)
\; & \,\mbox{if}\; x < S_{-}\\[0.2cm]
L(x)+\beta\rho\big( J_{\infty}(x-Z)\big) &\,\mbox{if}\; S_{-}\le x\le S_{+}\\[0.2cm]
(x-S_{+}) c_d + L(S_{+})+\beta \rho \big( J_{\infty}(S_{+}-Z)\big) \; &
\,\mbox{if}\; x> S_{+}.
\end{array}\right.$$
$J_\infty$ is convex.
  \item[b)] The stationary policy $(f^*,f^*,\ldots)$ is optimal
  with
  \begin{equation}\label{eq:cashbalance3}
f^*(x) := \left\{ \begin{array}{cl}
S_- \; &\, \mbox{if}\; x < S_-,\\
x & \,\mbox{if}\; S_- \le x\le S_+,\\
S_+\; & \,\mbox{if}\; x > S_+,
\end{array}\right.
\end{equation}
\end{enumerate}
\end{proposition}
Of course the switching points $S_-$ and $S^+$ depend on the choice of the risk measure $\rho$.
\end{example}

\bibliographystyle{amsplain}
\bibliography{Literature_MDP_recursive_risk_measures}

\end{document}